\documentclass[a4paper]{amsart}
\usepackage{etex}
\usepackage{enumerate}
\usepackage[english]{babel}

\usepackage{graphicx}
\usepackage{multicol}
\usepackage[bottom]{footmisc}
\usepackage[export]{adjustbox}
\usepackage{functan}
\usepackage[T1]{fontenc}
\usepackage{cancel}
\usepackage{calligra}
\usepackage{colortbl}
\usepackage{multirow}
\usepackage{enumitem}
\usepackage{varioref}
\usepackage{booktabs}
\usepackage{amscd}
\usepackage{color}
\usepackage{colortbl}
\usepackage{pstricks}
\usepackage{pst-all}
\usepackage{mparhack}
\usepackage{amssymb}
\usepackage{amsmath}
\usepackage{dsfont}
\usepackage{mathrsfs}
\usepackage{amsfonts}
\usepackage{amssymb}
\usepackage[mathcal]{euscript}
\usepackage[all,cmtip]{xy}
\usepackage{amssymb}
\usepackage{amsthm}
\theoremstyle{definition}
\newtheorem{definition}{Definition}[section]
\newtheorem{axiom}[definition]{Assumptions}

\newtheorem{remark}[definition]{Remark}
\newtheorem{notation}[definition]{Notation}

\theoremstyle{plain}

\newtheorem{theorem}[definition]{Theorem}
\newtheorem{proposition}[definition]{Proposition}
\newtheorem{lemma}[definition]{Lemma}

\numberwithin{equation}{section}

\def \alt96 {`}
\def \RN {\mathds{R}^N}
\def \R {\mathds{R}}

\def \N {\mathds{N}}

\def \G {\mathbb{G}}

\def \LL {{\mathcal{L}}}
\def \Ht {\mathcal{H}}
\def \d {\mathrm{d}}
\def \de {\partial}

\def \Lie {\mathrm{Lie}}

\def \LL {\mathcal{L}}

\begin{document}
 \title{Global Gaussian estimates for the heat kernel
 \\ of homogeneous sums of squares.}
 
 \author[S.\,Biagi]{Stefano Biagi}
 \author[M.\,Bramanti]{Marco Bramanti}
 
 \address{Dipartimento di Matematica
 \newline\indent Politecnico di Milano \newline\indent
 Via Bonardi 9, 20133 Milano, Italy}
 \bigskip
 
 \email{stefano.biagi@polimi.it}
 \email{marco.bramanti@polimi.it}



\date{\today}

 \begin{abstract}
Let $\mathcal{H}=\sum_{j=1}^{m}X_{j}^{2}-\partial_{t}$
be a heat-type operator in $\mathbb{R}^{n+1}$, where 
$X=\{X_{1},\ldots,X_{m}\}$ is a system of smooth H\"{o}rmander's vector fields in
$\mathbb{R}^{n}$, and every $X_{j}$ is homogeneous of degree $1$ with respect
to a family of non-isotropic dilations in $\mathbb{R}^{n}$, while no
underlying group structure is assumed. In this paper we prove
global (in space and time) upper and lower Gaussian estimates for the heat
kernel $\Gamma(t,x;s,y)$ of $\mathcal{H}$, in terms of the
Carnot-Carath\'{e}odory distance induced by $X$ on $\mathbb{R}^{n}$, as well
as global upper Gaussian estimates for the $t$- or $X$-derivatives of any
order of $\Gamma$. 
From the Gaussian bounds we derive the unique 
solvability of the Cauchy problem for a possibly unbounded continuous initial datum satisfying exponential growth at infinity. Also, we study the sol\-va\-bi\-li\-ty of the $\Ht$-Dirichlet problem 
on an arbitrary bounded domain. 
Finally, we establish a global scale-invariant 
Harnack inequality for non-negative solutions of
$\mathcal{H}u=0$.
\end{abstract}
\maketitle

\section{Introduction}
Let us consider a family $X=\left\{  X_{1},\ldots,X_{m}\right\}  $ of smooth
H\"{o}rmander's vector fields in $\mathbb{R}^{n}$ (precise definitions will be
given later). The study of the cor\-re\-spon\-ding heat-type operator
$$
\mathcal{H}:=\sum_{j=1}^{m}X_{j}^{2}-\partial_{t}\qquad\text{on $\mathbb{R}^{n+1}$}
$$
and its fundamental solution (heat kernel) has a long history and, by now, a
vast literature. The study of operators of the kind \alt96
sum of squares of H\"{o}rmander's vector fields',
$\mathcal{L}=\sum_{j=1}^{m}X_{j}^{2}$, as well as their evolutive counterpart,
$\mathcal{H}=\mathcal{L}-\partial_{t}$, is usually characterized by the
following dichotomy: \medskip

-\,\,\emph{local} properties of H\"{o}rmander operators of the kind $\mathcal{L}$
or $\mathcal{H}$ have been established for \emph{general }families of
H\"{o}rmander's vector fields $X_{1},\ldots,X_{m}$ (some cornerstones in this
context are \cite{Hor}, \cite{RS}, \cite{NSW}, \cite{SC}, \cite{J},
\cite{JSC}), while \medskip

-\,\,\emph{global }properties of $\mathcal{L}$ or $\mathcal{H}$ have been
established almost exclusively when the vector fields $X_{1},\ldots,X_{m}$ are
left invariant on some Lie group. \medskip

\noindent In particular, starting with the famous paper \cite{Folland} by Folland, a
rich theory exists under the as\-sump\-tion that $X_{1},\ldots,X_{m}$ be both left
invariant with respect to some group of translations, and homogeneous with
respect to some family of dilations (hence, $X_{1},\ldots,X_{m}$ are the
generators of a Carnot group $\mathbb{G}$ in $\mathbb{R}^{n}$). In that
context, the heat kernel has the form
\begin{equation}
\Gamma(t,x;s,y)=\gamma(y^{-1}\ast x,t-s)  \label{Gamma gamma}
\end{equation}
with $\gamma$ satisfying a two-sided Gaussian bound:%
$$
\frac{1}{Ct^{Q/2}}\exp\left(  -\frac{C\left\Vert x\right\Vert ^{2}}{t}\right)
\leq\gamma\left(  x,t\right)  \leq\frac{C}{t^{Q/2}}\cdot\exp\left(
-\frac{\left\Vert x\right\Vert ^{2}}{Ct}\right)
$$
for every $x\in\mathbb{G}$, $t>0$. Here $Q$ is the homogeneous dimension of
the group, and $\left\Vert \cdot\right\Vert $ is a homogeneous norm in
$\mathbb{G}$. Analogous upper bounds hold for the derivatives of every order:%
$$
\left\vert \partial_{t}^{m}X_{I}\gamma\left(  x,t\right)  \right\vert
\leq\frac{C}{t^{\left(  Q+\left\vert I\right\vert +2m\right)  /2}}\cdot
\exp\left(  -\frac{\left\Vert x\right\Vert ^{2}}{Ct}\right)
$$
where $X_{I}=X_{i_{1}}X_{i_{2}}...X_{i_{k}}$ with $i_{1},...,i_{k}\in\left\{
1,2,...,m\right\}  $, and $\left\vert I\right\vert =k$. 
The above Gaussian
bounds on Carnot groups are a special case of the more general results proved
for heat kernels corresponding to left invariant, but not necessarily
homogeneous, H\"{o}rmander's vector fields, by Varopoulos, Saloff-Coste,
Coulhon in \cite{VSCC}. They proved that for heat kernels on
\emph{nilpotent }Lie groups, a context where one still has
\eqref{Gamma gamma}, the function $\gamma$ satisfies a two sided bound
\begin{equation}%
\begin{split}
\frac{1}{C\big\vert B_{X}(0,\sqrt{t})\big\vert} &  \exp\bigg(-\frac{Cd_{X}%
^{2}(x,0)}{t}\bigg)\leq\gamma(x,t)\\[0.1cm]
&  \qquad\leq\frac{C}{\big\vert B_{X}(0,\sqrt{t})\big\vert}\exp\bigg(-\frac
{d_{X}^{2}(x,0)}{Ct}\bigg),
\end{split}
\label{gaussian SCSS}%
\end{equation}
and an upper bound on derivatives of every order:%
\begin{equation}
|\partial_{t}^{m}X_{I}\gamma(x,t)|\leq\frac{C}{\big\vert B_{X}(0,\sqrt
{t})\big\vert\,t^{(|I|+2m)/2}}\exp\bigg(-\frac{d_{X}^{2}(x,0)}{Ct}%
\bigg),\label{upper der SCSS}%
\end{equation}
where $d_{X}$ is the control distance induced by $X_{1},...,X_{m}$ and
$B_{X}(0,r)$ the cor\-re\-spon\-ding balls (see \cite[Thm. IV.4.2, Thm.
IV.4.3]{VSCC}). Also, they proved that on \emph{unimodular Lie groups with
polynomial volume growth}, that is satisfying
\[
c_{1}t^{D}\leq|B_{X}(0,\sqrt{t})|\leq c_{2}t^{D}\qquad\text{for $t\geq1$}%
\]
and some $D>0$, the above results \eqref{gaussian SCSS} and
\eqref{upper der SCSS} still hold (see \cite[Thm. VIII.2.7, Thm. 8.2.9]%
{VSCC}). For a different approach to Gaussian estimates in the context of Lie
groups with polynomial growth, see also the monograph \cite{DER} by Dungey,
ter Elst, Robinson. For the special case of Gaussian estimates on Carnot
groups, that we will explicitly exploit in this paper, we refer to the more
recent paper \cite{BLUpaper} by Bonfiglioli, Lanconelli, Uguzzoni. 

For a general system of H\"{o}rmander's vector fields, i.e., with \emph{no
un\-der\-ly\-ing group structure}, Gaussian bounds for the heat kernel
$$
\Gamma(t,x;s,y)=\gamma(t-s,x,y)
$$
have been proved by Jerison-Sanchez Calle \cite[Thms. 2, 3, 4]{JSC} in the
form:
\begin{align}
&  \frac{1}{C\big\vert B_{X}(x,\sqrt{t})\big\vert}
\exp\bigg(-\frac{Cd_{X}^{2}(x,y)}{t}\bigg) \leq
\gamma(t,x,y)
\leq\frac{C}{\vert B_{X}(x,\sqrt{t})\big\vert}%
\exp\bigg(-\frac{d_{X}^{2}(x,y)}{Ct}\bigg)  \label{Gaussian JS} \\[0.2cm]
&  \left\vert \partial_{t}^{m}X_{I}^{x}X_{J}^{y}\gamma(t,x,y)\right\vert
\leq\frac{C}{\left\vert B_{X}\left(  x,\sqrt{t}\right)  \right\vert t^{\left(
\left\vert I\right\vert +\left\vert J\right\vert +2m\right)  /2}}\exp\left(
-\frac{d_{X}\left(  x,y\right)  ^{2}}{Ct}\right)  \label{upper der JS}%
\end{align}
for every multiindices $I,J$, with $x,y$ ranging in a compact set and
$t\in(0,T)$. Using probabilistic techniques, Kusuoka-Stroock
have extended the above re\-sults to $x,y$ in $\mathbb{R}^{N}$ and $t\in\left(
0,T\right)  $, in \cite{KS1}, and later to $x,y$ in $\mathbb{R}^{N}$ and $t>0$
in \cite{KS2}. However, Kusuoka-Stroock require that the coefficients of the
vector fields belong to $C_{b}^{\infty}\left(  \mathbb{R}^{N}\right)  $. For
instance, vector fields with polynomial coefficients are not covered by their
theory (at least as far as global results are concerned). Related results
(under the same $C_{b}^{\infty}\left(  \mathbb{R}^{N}\right)  $ assumptions on
the vector fields) have been proved by L\'{e}andre in \cite{Le1}, \cite{Le2}.
Davies in \cite{D} has improved the constant in the exponent of the upper
bound in \eqref{Gaussian JS}, for a system of H\"{o}rmander's vector fields on
a compact manifold. \medskip

On the other hand, a general setting which allows to develop an interesting \emph{global theory},
with\-out assuming the existence of a group of translations, and allowing
un\-bound\-ed\-ness of the coefficients of $X_{1},...,X_{m}$ and their derivatives,
is that of H\"{o}rmander vector fields which are only assumed to be
$1$-homogeneous with respect to a family of non-isotropic dilations of the
form
$$
\delta_{\lambda}(x):=(\lambda^{\sigma_{1}}x_{1},\ldots,\lambda^{\sigma_{n}}x_{n}),
$$
where $1=\sigma_{1}\leq\ldots\leq\sigma_{n}$ are positive integers. In other
words,
$$
X_{j}(u\circ\delta_{\lambda})=(X_{j}u)\circ\delta_{\lambda}%
$$
for every $j=1,\ldots,m$, every $u\in C^{\infty}(\mathbb{R}^{n})
$ and every $\lambda>0$. Under this assumption (without any underlying group
structure), Biagi-Bonfiglioli in \cite{BiagiBonfiglioliLast} have built a
global homogeneous fundamental solution for $\mathcal{L}=\sum_{j=1}^{m}%
X_{j}^{2}$ and have studied some of its properties. The idea of this
construction is that, ac\-cor\-ding to a procedure originally devised by Folland
in \cite{Folland2} and adapted in \cite{BiagiBonfiglioliLast}, a system of
$1$\--ho\-mo\-ge\-neous H\"{o}rmander's vector fields can always be lifted to a higher
dimensional Carnot group where the corresponding sum of squares is known to
possess a global, left invariant, homogeneous fundamental solution. Saturating
this fundamental solution with respect to the added variables, in
\cite{BiagiBonfiglioliLast} a homogeneous fundamental solution for the
original operator is produced. More e\-xpli\-cit e\-sti\-ma\-tes for this kernel have
been established in \cite{BiBoBra1}, in terms of the distance
induced by the vector fields. The general strategy of
\cite{BiagiBonfiglioliLast} has been later 
implemented in \cite{BBHeat} for heat
operators corresponding to $1$-homogeneous vector fields, showing the
existence of a global, homogeneous, heat kernel, obtained by saturating the
heat kernel of a higher dimensional operator living on a Carnot group.

The aim of this paper is to prove sharp global explicit Gaussian estimates for
this heat kernel, in terms of the intrinsic distance induced by the vector
fields. More precisely, we will prove Gaussian estimates 
\eqref{Gaussian JS}-\eqref{upper der JS} for every $x,y\in\mathbb{R}^{n}$ and $t>0$, for heat
operators corresponding to $1$\--ho\-mo\-ge\-neous (but not left invariant)
H\"{o}rmander's vector fields (see Theorem \ref{thm.mainGaussianGamma}). 

Our global Gaussian bounds in particular allow to improve known results about
the Cauchy problem for this heat operator. In \cite[Thm. 4.1]{BBHeat} it is
proved that for every bounded continuous initial datum $f$ there exists one
and only one bounded solution to the Cauchy problem. We will prove that a
solution to the Cauchy problem actually exists, at least for small times, as
soon as the initial datum $f$ satisfies a growth condition of the kind
$$
\int_{\mathbb{R}^{n}}\vert f(y)\vert\exp
\big(-\mu\,d_{X}^{2}(y,0)\big)\,\d y<+\infty
$$
for some constant $\mu>0$. The solution is unique in the class of functions
sa\-ti\-sfy\-ing a condition
$$
\int_{0}^\tau\!\!\!\int_{\R^n}\exp\big(-\delta d_{X}^{2}(x,0)\big)
|u(t,x)|\,\d t\,\d x<+\infty
$$
for some $\delta>0$. Moreover, if $f$ satisfies a stronger bound of the kind
$$
 \int_{\mathbb{R}^{n}}\vert f(y)\vert 
 \exp\big(-\mu d_{X}^{\alpha}(y,0)\big)\,\d y < +\infty
 \qquad\text{for some $\alpha\in(0,2)$},
$$
 then the solution exists for all $t>0$
 (see Theorem
 \ref{thm.mainExistenceCauchy} and Proposition \ref{prop.growthfless2}).
 In Section \ref{sec.appDirichlet} we shall present an application
 of our global Gaussian estimates to the study of the
 $\Ht$-Dirichlet problem. In fact, by crucially exploiting these estimates,
 we shall show that it is possible to apply to our operators $\Ht$ the axiomatic
 approach developed in the series of papers
 \cite{Kogoj, LancoUguz, LancTralUguz, TralliUguz}; this will lead to
 some necessary and sufficient conditions
 for the regularity of boundary points of \emph{any} bounded open set $\Omega$. 
 Finally, in the last part of the paper we will prove a scale-invariant
 parabolic Harnack inequality for non-negative solutions of $\mathcal{H}u=0$
 (see Theorem \ref{thm.Harnack} in Section \ref{sec.Harnack}).  \medskip
 
We close this introduction with a few remarks about some related fields of
research. Gaus\-sian bounds for heat kernels have been studied, besides the
Eu\-cli\-dean setting, in the context of Riemannian manifolds. We can quote under
this respect the well-known paper \cite{LY} by Li-Yau where Gaussian bounds
are proved on manifolds with nonnegative curvature (see also the monograph
\cite{Gri} by Gri\-go\-r'yan and the references therein). Some extensions of these
geometric techniques to sub-Riemannian manifolds have been done, see e.g. the
paper \cite{BBG} by Baudoin, Bonnefont, Garofalo. Gaussian bounds have been
studied also in the abstract context of Dirichlet forms, see e.g. the papers
\cite{St2}, \cite{St3} by Sturm. These researches have made apparent a general
relation existing between the validity of Gaussian bounds for the heat kernel,
the validiy of global forms of Poincar\'{e}'s inequality and doubling
condition, and the validity of a parabolic Harnack inequality. For a
discussion of these general relations see also the monograph \cite{SalC} by
Saloff-Coste. In the context of homogeneous H\"{o}rmander vector fields
studied in the present paper, global forms of Poincar\'{e}'s inequality and
doubling condition are known, after \cite{BiBoBra1}. Therefore, our results
about Gaussian bounds and Harnack inequality are not unexpected. 
Nevertheless, we have not been able to find in the literature a precise theorem, directly
applicable to our context, implying our results. As far as we know, this is
the first case of global (in space and time) Gaussian estimates explicitly
proved, for
both the heat kernel and its derivatives of every order, in the context of
H\"{o}rmander's vector fields (with possibly unbounded coefficients) in
absense of an underlying group structure. 

\section{Assumptions and statements  of Gaussian bounds}
We denote by $\mathcal{X}(\mathbb{R}^{n})$ the Lie algebra of the smooth
vector fields on $\mathbb{R}^{n}$ (with $n\geq2$). Given a set $X\subseteq
\mathcal{X}(\mathbb{R}^{n})$, we indicate by $\mathrm{Lie}(X)$ the smallest
Lie sub-algebra of $\mathcal{X}(\mathbb{R}^{n})$ containing $X$. Finally, if
$Z\in\mathcal{X}(\mathbb{R}^{n})$ is a smooth vector field of the form
$$
Z=\sum_{j=1}^{n}a_{j}(x)\frac{\partial}{\partial x_{j}}\qquad
\text{ for some $a_{1},\ldots,a_{n}\in C^{\infty}(\mathbb{R}^{n})$}
$$
and if $x\in\mathbb{R}^{n}$, we denote by $Z(x)$ the
vector $(a_{1}(x), \ldots, a_n(x))\in\mathbb{R}^n$.
\begin{axiom}
\label{Assumptions}
Let $X=\{X_{1},\ldots,X_{m}\}$ {(}with $m\geq2${)} be a fixed
family of \emph{linearly independent} smooth vector fields in 
Euclidean space $\mathbb{R}^{n}$
sa\-ti\-sfy\-ing the fol\-lowing struc\-tu\-ral assumptions:

\begin{description}
\item[(H1)] there exists a family $\{\delta_{\lambda}\}_{\lambda>0}$ of
non-isotropic dilations of the form
\begin{equation} \label{eq.defidela}%
\delta_{\lambda}(x):=(\lambda^{\sigma_{1}}x_{1},\ldots,\lambda^{\sigma_{n}}x_{n}), 
\end{equation}
where $1=\sigma_{1}\leq\ldots\leq\sigma_{n}$ are positive integers, with
respect to which $X_{1},\ldots,X_{m}$ are homogeneous of degree $1$. This
means that
\begin{equation} \label{1-homoge}%
X_{j}(u\circ\delta_{\lambda})=(X_{j}u)\circ\delta_{\lambda} 
\end{equation}
for every $j=1,\ldots,m$, every $u\in C^{\infty}\left(  \mathbb{R}^{n}\right)
$ and every $\lambda>0$. We define the $\delta_{\lambda}$\--\emph{ho\-mo\-ge\-neous
dimension }of $\mathbb{R}^{n}$ as
\begin{equation} \label{eq.defiqH1}%
q:=\sum_{j=1}^{m}\sigma_{j}. 
\end{equation}
Note that $q\geq n.$

\item[(H2)] $X_{1},\ldots,X_{m}$ satisfy H\"{o}rmander's rank condition at
$x=0$, that is,
\begin{equation} \label{eq.assumptionHormander}%
\mathrm{dim}\big\{  Y(0):\,Y\in\mathrm{Lie}(X)\big\}  =n.
\end{equation}
\end{description}
\end{axiom}

\begin{remark} \label{rem.generalpropHt}
By combining assumptions {(H1)} and 
{(H2)}, it is not
difficult to recognize that H\"{o}r\-man\-der's rank condition is actually
satisfied \emph{at every point $x\in\mathbb{R}^{n}$}, that is,
$$
\dim\big\{Y(x):\,Y\in\operatorname{Lie}(X)\big\}
=n \qquad \text{for all $x\in\mathbb{R}^{n}$}
$$
(this is proved in \cite[Remark 3.2]{BiBoBra1}). Thus, by H\"{o}rmander's
Hypoellipticity Theo\-rem (see \cite{Hor}), both the operators $\mathcal{L}$ and
$\mathcal{H}$ are $C^{\infty}$-hypoelliptic in every open subset of
$\mathbb{R}^{n}$.
\end{remark}

In order to state our result, we first recall the following standard

\begin{definition}
[Carnot-Carath\'{e}odory distance]Let $Y=\{Y_{1},\ldots,Y_{h}\}$ be a family
of \emph{smooth} vector fields defined on some space $\mathbb{R}%
^{k}$. We assume that the $Y_{j}$'s satisfy H\"{o}rmander's rank condition at
every point of $\mathbb{R}^{k}$. The \emph{Car\-not\--Ca\-ra\-th\'{e}odory}
{(}CC, shortly{)} 
\emph{distance associated with $Y$} is defined as
$$
d_{Y}(x,y)=\inf\big\{r>0:\,\text{there exists $\gamma\in C(r)$ with
$\gamma(0)=x$ and $\gamma(1)=$}y\big\}  ,
$$
where $C(r)$ is the set of the absolutely continuous curves $\gamma
:[0,1]\rightarrow\mathbb{R}^{k}$ sa\-ti\-sfy\-ing 
{(}a.e.\,on $[0,1]${)}
$$
\gamma^{\prime}(t)=\sum_{j=1}^{h}a_{j}(t)\,Y_{j}(\gamma(t)),
\qquad\text{with $|a_{j}(t)| \leq r$ for all $j=1,\ldots,h$}.$$
We will denote by $B_{Y}\left(  x,\rho\right)  $ the metric ball $\left\{
y\in\mathbb{R}^{k}:d_{Y}(x,y)<\rho\right\}  .$
\end{definition}

Well-known results assure that under the above assumptions $d_{Y}(x,y)$ is
finite for every couple of points in $\mathbb{R}^{k}$ and that $(\mathbb{R}^{k},d_{Y})$ 
is a metric space; moreover, $d_{Y}$ is topologically, but not metrically,
equivalent to the Euclidean distance. \medskip

We can now state our main result:

\begin{theorem} \label{thm.mainGaussianGamma}
Let $X=\{X_{1},\ldots,X_{m}\}$ be a family of smooth
vector fields in $\mathbb{R}^{n}$ satisfying As\-sump\-tions \ref{Assumptions},
and
let $\mathcal{H}$ the heat-type operator
\begin{equation} \label{eq.defiHt}%
\mathcal{H}:=\mathcal{L}-\partial_{t}=
\sum_{j=1}^{m}X_{j}^{2}-\partial_{t}
\qquad\text{on $\mathbb{R}^{1+n}=\mathbb{R}_{t}\times\mathbb{R}_{x}^{n}$}.
\end{equation}
Moreover, let $\Gamma(t,x;s,y):=\gamma(t-s,x,y)$ be the global heat kernel of
$\mathcal{H}$, that will be precisely defined in
\eqref{defiGammaesplicitinaPRE1}. Then, the following facts hold.
\begin{itemize}
 \item[\emph{(i)}] 
 There exists a constant $\varrho>1$ such that
 \begin{equation} \label{eq.GuassianGamma}%
\begin{split}
&  \frac{1}{\varrho\,\vert B_{X}(x,\sqrt{t})\vert}\,\exp\bigg(
-\frac{\varrho\,d_{X}^{2}(x,y)}{t}\bigg)\leq\gamma(t,x,y)\\[0.2cm]
&  \qquad\quad\leq\frac{\varrho}{\vert B_{X}(x,\sqrt{t})\vert}\,
\exp\bigg(-\frac{d_{X}^{2}(x,y)}{\varrho\,t}\bigg),
\end{split}
\end{equation}
for every $x,y\in
\mathbb{R}^{n}$ and every $t>0$.

\item[\emph{(ii)}] For any nonnegative integers $k,r$ there exists $C
= C_{k,r}>0$ such
that%
\begin{equation} \label{Gassian_derivatives}%
\bigg\vert \bigg(\frac{\partial}{\partial t}\bigg)^{k}Y_{1}\cdots
Y_{r}\gamma(t,x,y)\bigg\vert \leq C\,\frac{t^{-(k+r/2)}}
{\vert B_{X}(x,\sqrt{t})\vert }\,\exp\bigg(-\frac{d_{X}^{2}(x,y)}{Ct}\bigg),
\end{equation}
for every choice of vector fields $Y_{1},\ldots,Y_{r}\in\{  X_{1}^{x},\ldots,X_{m}%
^{x},X_{1}^{y},\ldots,X_{m}^{y}\}$, and every choice of $x,y\in\mathbb{R}^{n}, t>0$.
\end{itemize}
\end{theorem}
 The results about the Cauchy problem for $\mathcal{H}$
 will be stated and proved in Section
 \ref{sec.Cauchy}, while our scale-invariant Harnack inequality 
 will be stated
 and proved in Section \ref{sec.Harnack}.

\section{Preliminaries and known results\label{sec.preliminari}}
\subsection{Carnot groups, lifting and construction of the heat kernel for $\mathcal{H}$}
We begin by recalling the definition of homogeneous Carnot group and some related notions
(see, e.g., \cite{BLUlibro} for an exhaustive treatment of this topic).\medskip

We say that $\mathbb{G}=(\mathbb{R}^{N},\ast,D_{\lambda})$ is a
\emph{homogeneous group }if $(\mathbb{R}^{N},\ast)$ is a Lie group (with group identity $e = 0$)
and if there exists a one-parameter family of group automorphisms $\{  D_{\lambda}\}_{\lambda>0}$ 
acting as in (H1).  We shall call the Lie group operation $\ast$ \emph{\alt96
translation'} and the automorphisms $D_{\lambda}$
\emph{\alt96 dilations'}. 

We say that a smooth vector field $X$ is \emph{left invariant} if, for every
$f\in C^\infty(\R^N)$, we have 
$$X\big(x\mapsto f(y\ast x)\big)
=(Xf)(y\ast x)\qquad\text{for all $x,y\in\mathbb{G}$}.$$
For $i=1,2,\ldots,n$, let $X_{i}\ $be the only left invariant vector field which
agrees at the origin with $\partial_{x_{i}}$. Assume that for some positive
integer $m<N$ we have that $X_{1},...,X_{m}$ are $1$-homogeneous (in the sense
of \eqref{1-homoge}) and that $X_{1},\ldots,X_{m}$ satisfy H\"{o}rmander's condition
as in (H2) (at the origin and then, by left invariance, at every point).
Then we say that 
$$\text{$\mathbb{G}$ is a \emph{Carnot group} and 
$X_{1},...,X_{m}$
are its \emph{generators}}.$$
A continuous function $\|\cdot\|:
\mathbb{G}\rightarrow [0,+\infty)$ is called a \emph{homogeneous norm} on $\mathbb{G}$ if there
exists $c>0$ such that, for every $u,v\in\mathbb{G}$, the following hold:
\begin{itemize}
 \item[(i)] $\|u\| =0$ if and only if $u = 0$;
 \item[(ii)] $\|D_{\lambda}(u)\|=\lambda\| u\|$ for every $\lambda>0$;
 \item[(iii)] $\|u\ast v\| \leq c\big(\|u\| +\|v\|\big)$;
 \item[(iv)] $\|u^{-1}\| \leq c\|u\|.$
\end{itemize}
If $X=\{X_{1},\ldots,X_{m}\}$ are the generators of a Carnot group
$\mathbb{G}$ and $d_{X}$ the Carnot-Carath\'{e}odory distance associated with
$X$, then $\|u\| =d_{X}(u,0)  $ is a
homogeneous norm on $\mathbb{G}$, further satisfying properties (iii)-(iv) with $c=1$.
\bigskip

A key information for the study of the operator $\mathcal{H}$ (and of its
associated heat kernel) is the dimension of the Lie algebra $\mathfrak{a}%
:=\mathrm{Lie}(X)$. Under our assumptions (H1)-(H2), 
it is easy to see that $\mathfrak{a}$ has finite
dimension: in fact, using \cite[Theorem A.11]{BBBook} and 
\cite[Proposition 1.3.10]{BLUlibro}), 
one has
$$
\mathfrak{a}=\bigoplus\limits_{k=1}^{\sigma_{n}}
\mathfrak{a}_{k}
$$
where $\mathfrak{a}_{1}:=\mathrm{span}\{X\} = \mathrm{span}\big\{X_{1},\ldots,X_{m}\big\}$
and
$$
\mathfrak{a}_{k}:=\mathrm{span}\{[Y,Z]:\,Y\in\mathfrak{a}_{1},\,Z\in\mathfrak{a}_{k-1}\}
\qquad (\text{for $k\geq 2$}).
$$
In particular, we obtain
\begin{equation} \label{eq.generalNgeqn}%
N=\mathrm{dim}(\mathfrak{a})\geq\mathrm{dim}\big\{Y(0):\,Y\mathfrak{a}\big\} =n. 
\end{equation}
As a consequence of \eqref{eq.generalNgeqn}, only the following two cases can occur.
\begin{itemize}
\item[(i)] $N=n$. In this case, by taking into account the $\delta_{\lambda}%
$\--ho\-mo\-ge\-nei\-ty of $X_{1},\ldots,X_{m}$, we can apply some
results in \cite{BonfLanc}, ensuring the existence of an operation $\ast$ on
$\mathbb{R}^{n}$ such that%
$$
\text{$\mathbb{F}=(\mathbb{R}^{n},\ast,\delta_{\lambda})$ is a homogeneous
Carnot group with $\operatorname{Lie}(\mathbb{F})=\mathfrak{a}$.}
$$
Hence, the vector fields $X_{1},\ldots,X_{m}$ are left invariant on
$\mathbb{F}$, and the operator $\mathcal{H}$ becomes the \emph{canonical heat
operator} on $\mathbb{R}\times\mathbb{F}$. This is a well-studied scenario, in
which all the results of this paper are well-known (see, for example, \cite{BLUpaper}).

\item[(ii)] $N>n$. In this case, instead, we derive from 
\cite[Theorem 1.4]{BBConstruction} that \emph{there cannot exist} any Lie-group structure in
$\mathbb{R}^{n}$ with respect to which $X_{1},\ldots,X_{m}$ are left
invariant. In particular, the operator $\mathcal{H}$ \emph{is not} a canonical
heat operator on some Carnot group.
\end{itemize}

In view of the above discussion, throughout the sequel, in the proof of our
results, we also make the following `dimensional' assumption.

\begin{description}
\item[(H3)] Using the notation $\mathfrak{a}=\operatorname{Lie}(X)$ and
$N=\dim(\mathfrak{a})$, we assume that
\begin{equation} \label{eq.assDima}%
p:=N-n\geq1. 
\end{equation}
\end{description}

\begin{remark}
Note that condition {(H3)} is not a further assumption that we require in order
for our results to be true. It is a further condition that is not restrictive
to assume within the proofs, because if our Assumptions \ref{Assumptions} hold
and {(H3)} is \emph{not} true, then our Theorem \ref{thm.mainGaussianGamma} is
already known.
\end{remark}

Even if assumption (H3) implies that $X_{1},\ldots,X_{m}$ cannot be left
invariant with respect to any Lie-group structure in $\mathbb{R}^{n}$, it is
proved in \cite{BiagiBonfiglioliLast} that the $X_{j}$'s can be lifted (in a
suitable sense) to vector fields $Z_{1},\ldots,Z_{m}$ which are left invariant
on a \emph{higher-dimensional} Carnot group:

\begin{theorem} [{Lifting, see {\cite[Theorem 3.1]{BiagiBonfiglioliLast}}}]\label{ThmA}
Let us suppose that as\-sump\-tions \emph{(H1)}-to-\emph{(H3)} are satisfied. Then, it is possible to
construct a homogeneous Carnot group 
$\mathbb{G}=(\mathbb{R}^{N},\ast,D_{\lambda})$ satisfying the following properties:

\begin{enumerate}
\item $\mathbb{G}$ has $m$ generators;

\item denoting the points of $\mathbb{R}^{N}$ as 
$u=(x,\xi)\in\mathbb{R}^{n}\times\mathbb{R}^{p}$, the family 
of dilations
$\{D_{\lambda}\}_{\lambda>0}$ takes the following `lifted' form:
\begin{equation}
D_{\lambda}(u)=D_{\lambda}(x,\xi)=\left(  \delta_{\lambda}(x),\delta_{\lambda
}^{\ast}(\xi)\right)  , \label{eq.expressionDlambda}%
\end{equation}
where $\delta_{\lambda}^{\ast}(\xi)=(\lambda^{\tau_{1}}\xi_{1},\ldots
,\lambda^{\tau_{p}}\xi_{p})$ for some integers $1\leq\tau_{1}\leq
\ldots\leq\tau_{p}$;

\item there exists a system of Lie-generators 
$\mathcal{Z}=\{Z_{1},\ldots,Z_{m}\}$ of $\operatorname{Lie}(\mathbb{G})$ s.t.
\begin{equation}
Z_{j}(x,\xi)=X_{j}(x)+R_{j}(x,\xi), \label{eq.liftinZjXj}%
\end{equation}
where the $R_i$'s are smooth vector fields operating only in the va\-ria\-bles
$\xi\in\mathbb{R}^{p}$, but with coefficient possibly depending on $(x,\xi)$. In
particular, $R_1,\ldots,R_m$ are $D_{\lambda}$-homogeneous of degree $1$.
\medskip
\end{enumerate}
\end{theorem}

\begin{notation}
Throughout the paper, we will handle points in the \alt96
original' space $\mathbb{R}^{n}$, and points in the
\alt96 lifted' space $\mathbb{R}^{N}$, according
to Theorem \ref{ThmA}. 
To this end, we shall use the notation
\begin{itemize}
\item $x,y,z,\ldots$ for points in $\mathbb{R}^{n}$;
\item $u  = (x,\xi), v= (y,\eta),\ldots$ for points in 
$\mathbb{R}^{N}\equiv\mathbb{R}^{n}\times\mathbb{R}^{p}$,
\end{itemize}
denoting by Greek letters the added variables in the lifting procedure. The
sca\-lar time variables will be denoted by letters $t,s,\tau$.
Moreover, we shall indicate by $d_{X}$ and $d_{\mathcal{Z}}$ the
Carnot-Carath\'{e}odory distances associated with $X$ and $\mathcal{Z}$,
respectively, and with $B_{X}(x,\rho),$ $B_{\mathcal{Z}}(u,\rho)$ the $d_{X}%
$-ball, $d_{\mathcal{Z}}$-ball, respectively, with centre $x\in\mathbb{R}^{n}%
$, $u\in\mathbb{R}^{N}$, and radius $\rho>0$.
\end{notation}
Since the lifted vector fields $Z_{1},\ldots,Z_{m}$ in Theorem \ref{ThmA} are
left invariant on $\mathbb{G}$, many properties of $\mathcal{H}_{\mathbb{G}}$
and its associated heat kernel are well-known. In fact, the following theorem holds.

\begin{theorem} [{{\cite[Theorems 2.1, 2.5]{BLUpaper}}}]\label{exTheoremB}
There exists a
function
$$
\gamma_{\mathbb{G}}:\mathbb{R}^{1+N}\rightarrow\mathbb{R},
$$
smooth away from the origin, such that
\begin{equation}
\Gamma_{\mathbb{G}}(t,u;s,v):=\gamma_{\mathbb{G}}\left(  t-s,v^{-1}\ast
u\right)  \label{eq.defiGammaGHtG}%
\end{equation}
is the global heat kernel of
$\mathcal{H}_{\mathbb{G}}=\mathcal{L}_{\mathbb{G}}-\partial_{t}$; 
this means, precisely, that
\begin{itemize}
\item for every fixed $(t,z)\in\mathbb{R}^{1+N}$, one has 
$\Gamma_{\mathbb{G}}(t,z;\cdot)\in L_{\operatorname{loc}}^{1}(\mathbb{R}^{1+N})$;

\item for every $\varphi\in C_{0}^{\infty}(\mathbb{R}^{1+N})$
and every $(t,u)\in\mathbb{R}^{1+N}$, one has
\begin{align*}
&  \mathcal{H}_{\mathbb{G}}\bigg(
 \int_{\mathbb{R}^{1+N}}\Gamma_{\mathbb{G}}(t,u;s,v)\varphi(s,v)\,\d s\,\d v\bigg) 
 \\[0.1cm]
 & \qquad =\int_{\mathbb{R}^{1+N}}\Gamma_{\mathbb{G}}(t,u;s,v)
 \mathcal{H}_{\mathbb{G}}\varphi(s,v)\,\d s\,\d v = -\varphi(t,u).
\end{align*}
\end{itemize}
Furthermore, $\gamma_{\mathbb{G}}$ satisfies the following properties:
\begin{itemize}
\item[\emph{(i)}] $\gamma_{\mathbb{G}}\geq0$ and $\gamma_{\mathbb{G}}(t,u)=0$
if and only if $t\leq0$;

\item[\emph{(ii)}] $\gamma_{\mathbb{G}}(t,u)=\gamma_{\mathbb{G}}(t,u^{{-1}})$
for every $(t,u)\in\mathbb{R}^{1+N}$;

\item[\emph{(iii)}] for every $\lambda>0$ and every $(t,u)$, we have
$$
\gamma_{\mathbb{G}}(\lambda^{2}t,D_{\lambda}(u)) =
 \lambda^{-Q}\,\gamma_{\mathbb{G}}(t,u),
$$
where $Q$ is the homogeneous dimension of the group $\mathbb{G}$, that is,
\begin{equation} \label{eq.dimQGlifting}
Q:=q+q^{\ast},\qquad\text{with $q$ as in \eqref{eq.defiqH1} and 
$q^{\ast}:=\textstyle\sum_{k=1}^{p}\tau_{k}$}; 
\end{equation}

\item[\emph{(iv)}] $\gamma_{\mathbb{G}}$ vanishes at infinity, that is,
$\gamma_{\mathbb{G}}(t,u)\rightarrow0$ as $|(t,u)|\rightarrow+\infty$;

\item[\emph{(v)}] for every $t>0$, we have
$$
\int_{\mathbb{R}^{N}}\gamma_{\mathbb{G}}(t,u)\,\d u=1.
$$
\end{itemize}
Finally, the following Gaussian estimates for $\gamma_{\mathbb{G}}$ hold:
\begin{itemize}
\item[\emph{(a)}] there exists a constant $\mathbf{c}\geq1$, only depending on
$\mathbb{G}$ and $\mathcal{Z}$, s.t.
\begin{equation} \label{eq.Gaussest}%
\mathbf{c}^{-1}\,t^{-Q/2}\,\exp\bigg(-\frac{\mathbf{c}\,\|u\|^{2})}{t}\bigg)
 \leq\gamma_{\mathbb{G}}(t,u)\leq\mathbf{c}\,t^{-Q/2}\,
 \exp\bigg(-\frac{\|u\|^{2}}{\mathbf{c}\,t}\bigg), 
\end{equation}
for every $u\in\mathbb{R}^{N}$ and every $t>0$.

\item[\emph{(b)}] for every nonnegative integers $h,k$ there exists a constant
$\widehat{\mathbf{c}}>0$ s.t.
\begin{equation}  \label{eq.Gaussestder}%
\bigg\vert Z_{i_{1}}...Z_{i_{h}}
\bigg(\frac{\partial}{\partial {t}}\bigg)^{k}\gamma_{\mathbb{G}}(t,u)\bigg\vert 
\leq\widehat{\mathbf{c}}\,t^{-(Q+h+2k)/2}\,\exp
\bigg(-\frac{\|u\|^{2}}{\widehat{\mathbf{c}}\,t}\bigg)
\end{equation}
for any $u\in\mathbb{R}^{N}$, any $t>0$ and every choice of 
$i_1,\ldots, i_{h}\in \{  1,\ldots,m\}  .$
\end{itemize}
\end{theorem}
Now, the `lifting property' \eqref{eq.liftinZjXj} contained in Theorem
\ref{ThmA} easily implies that
\begin{equation} \label{eq.liftinHHG}%
\mathcal{H}_{\mathbb{G}}
\big((t,x)\mapsto u(t,\pi(x))\big)=
(\mathcal{H}u)(t,\pi(x)), \qquad\text{for all $u\in C^{2}(\mathbb{R}^{n})$}
\end{equation}
where $\pi$ is the projection of $\mathbb{R}^{N}=\mathbb{R}^{n}\times
\mathbb{R}^{p}$ on $\mathbb{R}^{n}$. By combining \eqref{eq.liftinHHG} with
Theorem \ref{exTheoremB}, it is proved in \cite{BBHeat} the following result.

\begin{theorem} [{{\cite[Theorem 1.4]{BBHeat}}}]\label{ThmC}
Let $X=\{X_{1},\ldots,X_{m}\}$ be
a set of smooth vector fields on $\mathbb{R}^{n}$ satisfying
axioms \emph{(H1)}-to-\emph{(H3)}, and let $\mathcal{H}$ be
the heat-type operator defined in \eqref{eq.defiHt}.
Moreover, let $\mathbb{G}=(\mathbb{R}^{N},\ast,D_{\lambda})$ and
$\mathcal{Z}=\{Z_{1},\ldots,Z_{m}\}$ be as in Theorem \ref{ThmA}. \medskip

Then, if
$\gamma_{\mathbb{G}}$ is as in Theorem \ref{exTheoremB}, the following facts hold.

\begin{itemize}
\item[\emph{(i)}] The function $\Gamma$ defined by
\begin{equation}\label{defiGammaesplicitinaPRE1}%
\Gamma(t,x;s,y):=\gamma(t-s,x,y):=\int_{\mathbb{R}^{p}}\gamma_{\mathbb{G}}
\big(t-s,(y,0)^{-1}\ast(x,\eta)\big)\,\d\eta, 
\end{equation}
is the global heat kernel of $\mathcal{H}$. This means, precisely, that
\begin{itemize}
\item[$\mathrm{(i)_{1}}$] for any fixed $(t,x)\in\mathbb{R}^{1+n}$, we have
$\Gamma(t,x;\cdot)\in L_{\operatorname{loc}}^{1}(\mathbb{R}^{1+n})$;

\item[$\mathrm{(i)_{2}}$] for every $\varphi\in C_{0}^{\infty}(\mathbb{R}^{1+n})$
and every $(t,x)\in\mathbb{R}^{1+n}$, we have
\begin{equation*}
\begin{split}
& \mathcal{H}\bigg(\int_{\mathbb{R}^{1+n}}\gamma(t-s,x,y)\varphi(s,y)\,\d s\,\d y\bigg) \\[0.1cm]
& \qquad
= \int_{\mathbb{R}^{1+n}}\gamma(t-s,x,y)\,\mathcal{H}\varphi(s,y)\,\d s\,\d y = 
-\varphi(t,x).
\end{split}
\end{equation*}
\end{itemize}

\item[\emph{(ii)}] There exists a constant $\mathbf{c}\geq1$ such that
\begin{equation}\label{eq.estimGammaquasiGauss}%
\begin{split}
& \mathbf{c}^{-1}\,t^{-Q/2}\,\int_{\mathbb{R}^{p}}\exp\bigg(
-\frac{\mathbf{c}\,\|(y,0)^{-1}\ast(x,\eta)\|^{2}}{t}\bigg)\,\d\eta
\leq\gamma(t,x,y)\\[0.2cm]
&  \qquad\qquad\leq\mathbf{c}\,t^{-Q/2}\,\int_{\mathbb{R}^{p}}\exp\bigg(
-\frac{\|(y,0)^{-1}\ast(x,\eta)\|^{2}}{\mathbf{c}\,t}\bigg)\,\d\eta,
\end{split}
\end{equation}
for
every $x,y\in\mathbb{R}^{n}$ and every $t>0$.

\item[\emph{(iii)}] $\gamma\geq0$ and
$$
\text{$\gamma(t,x,y)=0$ if and only if $t\leq 0$}.
$$

\item[\emph{(iv)}] $\gamma$ is symmetric in the space variables, i.e.%
$$
\gamma(t,x,y)=\gamma(t,y,x)\qquad \text{for every $x,y\in\mathbb{R}^{n}$ and every $t>0$}.
$$

\item[\emph{(v)}] $\Gamma$ is smooth out of the diagonal of $\mathbb{R}^{1+n}\times\mathbb{R}^{1+n}$.

\item[\emph{(vi)}] For every fixed $(t,x)\in\mathbb{R}^{1+n},$ with $t>0$, we
have
$$
\int_{\mathbb{R}^{n}}\gamma(t,x,y)\,\d y=1.
$$

\item[\emph{(vii)}] If $\varphi\in C_{b}^{0}(\mathbb{R}^{n})$, then the
function
$$
u(t,x):=\int_{\mathbb{R}^{n}}\gamma(t,x,y)\,\varphi(y)\,\d y
$$
defined for $(t,x)\in\Omega=(0,+\infty)\times\mathbb{R}^{n}$ is the unique
bounded classical solution of the ho\-mo\-ge\-neo\-us Cauchy problem
for $\mathcal{H}$, that is,
$$
\begin{cases}
\mathcal{H}u=0 & \text{in $\Omega$}\\
u(0,x)=\varphi(x) & \text{for $x\in\mathbb{R}^{n}$.}%
\end{cases}
$$

\item[\emph{(viii)}] The function $\Gamma^{\ast}(t,u;s,v)=\Gamma
(s,v;t,u)$ is the global heat kernel of the \emph{(}formal\emph{)} adjoint operator
$\mathcal{H}^{\ast}:=\mathcal{L}+\partial_{t}$, and satisfies dual statements
with respect to \emph{(i)}.
\end{itemize}
\end{theorem}

In the above theorem, $\|\cdot\| $ is \emph{any} homogeneous norm
on $\mathbb{G}$. 
\begin{remark} \label{Remark switched Gaussian}
Points {(ii)} and 
{(iv)} in the above theorem also
imply that, with the same constant $\mathbf{c}\geq1$ as in {(ii)}, for all
$x,y\in\mathbb{R}^{n}$ and any $t>0$ one has
\begin{equation}\label{eq.estimGammaquasiGauss2}%
\begin{split}
& \mathbf{c}^{-1}\,t^{-Q/2}\,\int_{\mathbb{R}^{p}}\exp\bigg(
-\frac{\mathbf{c}\,\|(x,0)^{-1}\ast(y,\eta)\|^{2}}{t}\bigg)\,\d\eta
\leq\gamma(t,x,y)\\[0.2cm]
&  \qquad\qquad\leq\mathbf{c}\,t^{-Q/2}\,\int_{\mathbb{R}^{p}}\exp\bigg(
-\frac{\|(x,0)^{-1}\ast(y,\eta)\|^{2}}{\mathbf{c}\,t}\bigg)\,\d\eta,
\end{split}
\end{equation}
{(}with the switched roles of $x,y$ in the Gaussians{)}. It will be sometimes
con\-ve\-nient to use \eqref{eq.estimGammaquasiGauss} in this alternative form.
\end{remark}

\subsection{Review of known results on the CC distance} \label{app.CCdistanceremind}

Throughout the sequel, we will handle two distinct families of
H\"{o}rmander's vector fields, each one inducing a Carnot-Carath\'{e}odory
 distance:
 \begin{itemize}
 \item the original
family of vector fields $X_{1},...,X_{m}$, defined in $\mathbb{R}^{n}$, and
satisfying (H1)-to-(H3);
\item the lifted vector fields
$Z_{1},...,Z_{m}$, defined on the higher dimensional Carnot group $\mathbb{G}$
in $\mathbb{R}^{N}$.
\end{itemize}
 Both the $X_{i}$'s and the $Z_{i}$'s are $1$-homogeneous
with respect to suitable dilations, which implies some properties of the
distances and the corresponding balls. The $Z_{i}$'s are also left invariant,
which implies more properties for the corresponding distance. Finally, the
$Z_{i}$'s are a lifting of the $X_{i}$'s. The next proposition collects the
basic properties which follow from these facts. 

\begin{proposition} \label{Prop properties distances}
With the previous notation and assumptions
about the sy\-stems of vector fields $X$ and $\mathcal{Z}$, the following properties
hold.
\medskip

\emph{(i)}\,\,Homogeneity:
$$
\begin{array}{ll}
d_{X}(\delta_{\lambda}(x),\delta_{\lambda}(y))=\lambda\,d_{X}(x,y) &
\text{for all $x,y\in \mathbb{R}^{n}$ and $\lambda>0$} \\[0.15cm]
d_{\mathcal{Z}}(D_{\lambda}(u),D_{\lambda}(v))=\lambda\,d_{\mathcal{Z}}(u,v)
& \text{for all $u,v\in \mathbb{R}^{N}$ and $\lambda>0$} \\[0.15cm]
\delta_{\lambda}\big(B_{X}(x,\rho)\big)
= B_{X}\big(\delta_{\lambda}(x),\lambda\rho\big) & 
\text{for all $x\in\mathbb{R}^{n}$ and $\lambda,\rho>0$} \\[0.15cm]
D_{\lambda}\big(B_{\mathcal{Z}}(u,\rho)\big)  = 
B_{\mathcal{Z}}\big(D_{\lambda}(u),\lambda\rho\big) & 
\text{for all $u\in\mathbb{R}^{N}$ and $\lambda,\rho>0$}
\end{array}
$$

\emph{(ii)}\,\,Left invariance:
$$
\begin{array}{ll}
d_{\mathcal{Z}}(u,v)=d_{\mathcal{Z}}(u\ast w,v\ast w) & \text{for all
$u,v,w\in\mathbb{R}^{N}$} \\[0.15cm]
u\ast B_{\mathcal{Z}}(v,\rho)  =B_{\mathcal{Z}}(u\ast v,\rho)
& \text{for all $u,v\in\mathbb{R}^{N}$ and $\rho>0$}
\end{array}
$$

\emph{(iii)}\,\,Projection:
$$
\begin{array}{ll}
d_{X}(x,y)\leq d_{\mathcal{Z}}((x,\xi),(y,\eta)) & \text{for
all $(x,\xi),(y,\eta)\in\mathbb{R}^{N}=\mathbb{R}^{n}\times\mathbb{R}^{p}$} \\[0.15cm]
\pi\big(B_{\mathcal{Z}}((x,\xi),\rho)\big)  = 
B_{X}(x,\rho) & \text{for all $(x,\xi)\in\mathbb{R}^{N}$ and $\rho>0$}
\end{array}
$$
where $\pi$ is the projection from $\mathbb{R}^{N}=\mathbb{R}^{n}%
\times\mathbb{R}^{p}$ into $\mathbb{R}^{n}$. In particular, since $\pi$ is
surjective, the last equality in \emph{(iii)} means that
\begin{equation}\label{item.pjectionBY}%
\text{$\forall\,\,y\in B_{X}(x,\rho),\,\,\xi\in\mathbb{R}^{p}\,\,\,\exists\,\,\,
\eta
\in\mathbb{R}^{p}$ s.t.\,$(y,\eta)\in B_{\mathcal{Z}}(x,\xi) ,\rho)$}.
\end{equation}

\emph{(iv)} Volume of $\mathcal{Z}$-balls: setting $\omega_{Q}= |B_{\mathcal{Z}}(0,1)|$,
we have
\begin{equation} \label{eq.measureBYCarnot}
\begin{array}{c}
\text{$|B_{\mathcal{Z}}(u,\rho)| = |B_{\mathcal{Z}}(0,\rho)| =
\omega_{Q}\,\rho^{Q}$ for all $u\in\mathbb{R}^{N}$ and $\rho>0$}. 
\end{array}
\end{equation}

\emph{(v)} Homogeneous norm: if we let
$$
\| u\| =d_{\mathcal{Z}}(u,0) \qquad \text{ for every $u\in\mathbb{R}^{N}$},
$$
then $\| \cdot\| $ is a homogeneous norm, and we also have
$$
d_{\mathcal{Z}}(u,v)  =\| v^{-1}\ast u\|
=\| u^{-1}\ast v\|\qquad \text{ for every $u,v\in\mathbb{R}^{N}$}.
$$
Throughout the following, the symbol $\|\cdot\|$ in 
$\mathbb{R}^{N}$ will always denote this special norm.
\end{proposition}
The proof of Proposition \ref{Prop properties distances} can
be found in \cite{BiagiBonfiglioliLast}, or is immediate. \medskip

A much deeper result describes the \emph{volume of }$X$\emph{-balls}. The
following theo\-rem specializes a celebrated result by Nagel, Stein and Wainger
\cite{NSW} to the case of our $1$-homogeneous vector fields $X$ (for a proof
see \cite[Theorem B]{BiBoBra1}):

\begin{theorem}
\label{thm.NSWglobal} Let $X=\{X_{1},\ldots,X_{m}\},\,n$ and $q$ be as before. Then,
there exist constants $\gamma_{1},\gamma_{2}>0$
such that, for every $x\in\mathbb{R}^{n}$ and every $\rho>0$, one has the estimates
\begin{equation} \label{eq.NSWmodificata}%
\gamma_{1}\,\sum_{j=n}^{q}f_{j}(x)\,\rho^{j}\leq\vert {B_{X}(x,\rho
)}\vert \leq\gamma_{2}\,\sum_{j=n}^{q}f_{j}(x)\,\rho^{j}.
\end{equation}
Here, 
the functions $f_{k},\ldots,f_{q}:\mathbb{R}^{n}\rightarrow\mathbb{R}$ satisfy the following properties:

\begin{enumerate}
\item $f_{k},\ldots,f_{q}$ are continuous and non-negative on $\mathbb{R}^{n}$;

\item for every $j\in\{n,\ldots,q\}$, the function $f_{j}$ is 
$\delta_{\lambda}$-homogeneous of degree $q-j$.
\end{enumerate}
In particular, $f_{q}(x)$ is constant in $x$ and strictly positive.
\end{theorem}

\begin{remark} \label{rem.doublingglobale}
 From estimate \eqref{eq.NSWmodificata} it can be
 easily derived the following notable fact: \emph{for any $x\in\mathbb{R}^{n}$
 and any $0<r<\rho$, one has}
 \begin{equation} \label{eq.doublingglobalegeneral}%
 \gamma_{1}\,\bigg(\frac{\rho}{r}\bigg)^{n}\leq\frac{|{B_{X}(x,\rho)}|}
 {|{B_{X}(x,r)}|}\leq\gamma_{2}\,\bigg(\frac{\rho}{r}\bigg)^{q},
 \end{equation}
 In particular, the following \emph{global} doubling property holds: 
 \begin{equation} \label{eq.doublingesplicita}%
 | {B_{X}(x,2\rho)}| \leq 2^{q}\gamma_2\,|
 {B_{X}(x,\rho)}| \qquad \text{for all $x\in\mathbb{R}^{n}$ and $\rho>0$}.
 \end{equation}
 \end{remark}
The above facts easily imply that the
function
$$
\frac{1}{| B_{X}(x,\sqrt{t})| }\,\exp\bigg(-\frac{d_{X}^{2}(x,y)}{t}\bigg)
$$
(which plays a key role in our estimates) is not so asymmetric in $x,y$ 
as could seem. More precisely, we have the following proposition.

\begin{proposition} \label{Remark equivalent gaussian}
For every $\theta > 0$ there exists a constant $C_1 > 0$ such that
\begin{equation} \label{equivalent gaussian}%
\frac{1}{| B_{X}(y,\sqrt{t})| }\,
\exp\bigg(-\frac{d_{X}^{2}(x,y)}{\theta t}\bigg) \leq \frac{C_{1}}
{|B_{X}(x,\sqrt{t})|}\,\exp\bigg(-\frac{d_{X}^{2}(x,y)}{C_{1}\theta t}\bigg),  
\end{equation}
for every
$x,y\in\mathbb{R}^{n}$ and every $t>0$.
\end{proposition}

\begin{proof}
To prove this, let us distinguish two cases.
\begin{itemize}
\item If $d_X(x,y) \leq\sqrt{t}$, we infer from \eqref{eq.doublingesplicita}
that
$|B_{X}(y,\sqrt{t})| $ and $|B_{X}(x,\sqrt{t})|$ are equivalent, 
and thus the above inequality holds.

\item If $d_X(x,y)  >\sqrt{t}$, then by \eqref{eq.doublingglobalegeneral}
and \eqref{eq.doublingesplicita} 
we have
\begin{align*}
\frac{1}{|B_{X}(y,\sqrt{t})|} &  \leq
\frac{\gamma_{2}}{|B_{X}(y,d_X(x,y))|}\cdot\left(
\frac{d_X(x,y)}{\sqrt{t}}\right)^{q} \\[0.15cm]
& \leq\frac{2^q\,(\gamma_2)^2}{\big|B_{X}\big(x,d_X(x,y)\big)\big|}\cdot
\left(\frac{d_X(x,y)}{\sqrt{t}}\right)^{q} \\[0.15cm]
& \leq \frac{2^q\,(\gamma_2)^2}{\gamma_1}\cdot\,\frac{1}{|B_{X}(x,\sqrt{t})|}\cdot
\left(\frac{d_X(x,y)}{\sqrt{t}}\right)^{q-n}.
\end{align*}
From this, we readily obtain
\eqref{equivalent gaussian}
(see, e.g., \eqref{eq.stimaExpAlto}).
\end{itemize}
This ends the proof.
\end{proof}
Another deep known result that will play a key role in our estimates is the
following `global' version of a well-known result by Sanch\'{e}z-Calle
\cite{SC} (see also \cite[Lemma 3.2]{NSW} and \cite{J}), which compares the
volumes of $B_{X}(x,\rho)$ and $B_{\mathcal{Z}}((x,\xi),\rho)$.
For a proof of this result see \cite[Theorem C]{BiBoBra1}.

\begin{theorem} \label{thm.globalSanchez}
Under the previous assumptions and notation, there exist constants
$\kappa\in(0,1)$ and $c_{1},c_{2}>0$ such that, for every $x\in\mathbb{R}^{n}%
$, every $\xi\in\mathbb{R}^{p}$ and every $\rho>0$ one has the
estimates:
\begin{align}
\big|\{\eta\in\mathbb{R}^{p}:(y,\eta)\in B_{\mathcal{Z}}((x,\xi
),\rho)\}\big|  &  \leq c_{1}\frac{| B_{\mathcal{Z}}((x,\xi),\rho)|}
{|B_{X}(x,\rho)\vert }%
,\quad\text{for all $y\in\mathbb{R}^{n}$},\label{eq.SanchezI}\\[0.3cm]
\big|\{\eta\in\mathbb{R}^{p}:(y,\eta)\in B_{\mathcal{Z}}((x,\xi
),\rho)\}\big| &  \geq c_{2}\frac{\vert B_{\mathcal{Z}}%
((x,\xi),\rho)\vert }{\vert B_{X}(x,\rho)\vert }%
,\quad\text{for all $y\in B_{X}(x,\kappa\rho)$}. \label{eq.SanchezII}%
\end{align}
\end{theorem}
We wish to stress that Theorems \ref{thm.NSWglobal} and
\ref{thm.globalSanchez} contain \emph{global }results, adapted to our context
of homogeneous vector fields. In contrast with this, the original versions of
these results, contained in \cite{NSW}, \cite{SC}, and related to general
systems of H\"{o}rmander's vector fields, express \emph{local }results.

\section{Gaussian estimates for $\Gamma$\label{sec.estimGamma}}
The aim of this section is to prove upper/lower Gaussian estimates for the
global heat kernel $\Gamma(t,x;s,y)$ of $\mathcal{H}$ (or, equivalently,
for $\gamma(t,x,y)$) as defined in \eqref{defiGammaesplicitinaPRE1}). 
Broadly put, our approach is the following: on account of
\eqref{eq.estimGammaquasiGauss2}, 
we already know that $\Gamma$ satisfies the `quasi-Gaussian' estimates
$$
\gamma(t,x,y)\approx t^{-Q/2}\,\int_{\mathbb{R}^{p}}\exp
\bigg(-\frac{\|(x,0)^{-1}\ast(y,\eta)\|^{2}}{t}\bigg)\,\d\eta,
$$
where $Q$ is as in \eqref{eq.dimQGlifting}; 
we then derive `pure' Gaussian
estimates for $\Gamma$ by showing that,
for any $x,y\in\mathbb{R}^{n}$ and any $t>0$, one has
\begin{equation} \label{eq.aimGaussiantoprove}%
t^{-Q/2}\,\int_{\mathbb{R}^{p}}\exp
\bigg(-\frac{\|(x,0)^{-1}\ast(y,\eta)\|^{2}}{t}\bigg)\,\d\eta
\approx\frac{1}{|B_{X}(x,\sqrt{t})|}\cdot\exp\bigg(-\frac{d_{X}^{2}(x,y)}{t}\bigg).
\end{equation}
To begin with, for a future reference, we state the following lemma.
\begin{lemma}
\label{lem.stimeExp}
The following estimates hold true:
\begin{itemize}
\item[\emph{(i)}] for every $\nu>0$ and $\delta\in(0,1)$ there exists $c>0$
such that%
\begin{equation}
\tau^{\nu}e^{-\tau^{2}}\leq c\,e^{-\delta\tau^{2}}\quad\text{for every $\tau
\geq0$}; \label{eq.stimaExpAlto}%
\end{equation}
\item[\emph{(ii)}] for every positive $\nu,\theta\ $there exists $c>0$ such
that%
\begin{equation}
\tau^{-\nu}\geq c\,e^{-\theta\tau^{2}}\quad\text{for every $\tau>0$}.
\label{eq.stimaExpBasso}%
\end{equation}
\end{itemize}
\end{lemma}
We then proceed by proving \eqref{eq.aimGaussiantoprove}, and we start with
the upper estimate.
\begin{proposition} \label{prop.upperestimateGamma}
There exists a 
constant $\kappa>1$ such that
\begin{equation}\label{eq.mainUpperestim}%
t^{-Q/2}\,\int_{\mathbb{R}^{p}}\exp
\bigg(-\frac{\|(x,0)^{-1}\ast(y,\eta)\|^{2}}{t}\bigg)\,\d\eta 
\leq\frac{\kappa}{|B_{X}(x,\sqrt{t})|}\,\exp\left(  -\frac{d_{X}^{2}(x,y)}{2t}\right),
\end{equation}
for every
$x,y\in\mathbb{R}^{n}$ and every $t>0$.
\end{proposition}

\begin{proof}
Let $x,y\in\mathbb{R}^{n}$ be arbitrarily fixed, and let $t>0$.
\medskip

\textsc{Case I:} $d_{X}(x,y)>\sqrt{t}$. In this case, for every $n=0,1,2,...$,
we define
\begin{equation} \label{eq.defisetAncase1upper}%
A_{n}:=\big\{  \eta\in\mathbb{R}^{p}:\,2^{n}d_{X}(x,y)\leq\|
(x,0)^{-1}\ast(y,\eta)\| <2^{n+1}d_{X}(x,y)\big\},
\end{equation}
and we observe that, by Proposition \ref{Prop properties distances}-(iii), 
it holds
$\mathbb{R}^{p}=\cup_{n\geq0}A_{n}$. Hence,
\begin{align*}
&  \int_{\mathbb{R}^{p}}\exp\left(-\frac{\|(x,0)^{-1}\ast
(y,\eta)\|^{2}}{t}\right)\d\eta
 = \sum_{n=0}^{+\infty}\int_{A_{n}}
\exp\left(-\frac{\|(x,0)^{-1}\ast(y,\eta)\|^{2}}{t}\right)\d\eta\\[0.2cm]
&  \quad\leq\sum_{n=0}^{+\infty}\exp\left(-\frac{2^{2n}d_{X}^{2}(x,y)}{t}\right)
\cdot|A_{n}|\\[0.2cm]
&  \quad\leq \sum_{n=0}^{+\infty}\exp\left(-\frac{2^{2n}d_{X}^{2}(x,y)}{t}\right)
\cdot\big\vert
\big\{\eta\in\mathbb{R}^{p}:\,(y,\eta)\in
B_{\mathcal{Z}}\big((x,0),2^{n+1}d_{X}(x,y)\big)\big\}\big\vert
\\[0.2cm]
& \quad =:(\bigstar).
\end{align*}
Next, by combining Theorem \ref{thm.globalSanchez} and
\eqref{eq.measureBYCarnot}, for every $n\geq0$ we have
\begin{align*}
&  \big\vert
\big\{\eta\in\mathbb{R}^{p}:\,(y,\eta)\in
B_{\mathcal{Z}}((x,0),2^{n+1}d_{X}(x,y))\big\}\big\vert \leq
c_{1}\,\frac{\big\vert B_{\mathcal{Z}}\big((x,0),2^{n+1}d_{X}(x,y)\big)\big\vert}
{\big\vert B_{X}\big(x,2^{n+1}d_{X}(x,y)\big)\big\vert}\\[0.2cm]
&  \qquad=c_{1}\omega_{Q}\,\frac{2^{Q(n+1)}d_{X}^{Q}(x,y)}
{\big|B_{X}\big(x,2^{n+1}d_{X}(x,y)\big)\big\vert }
\leq c_{1}\omega_{Q}\,\frac{2^{Q(n+1)}d_{X}^{Q}(x,y)}
{|B_{X}(x,d_{X}(x,y))|}\\[0.2cm]
&  \qquad\leq c_{1}\omega_{Q}\,\frac{2^{Q(n+1)}d_{X}^{Q}(x,y)}
{|B_{X}(x,\sqrt{t})\vert },
\end{align*}
since $d_{X}(x,y)>\sqrt{t}$. As a consequence, we obtain
\begin{align*}
(\bigstar) &  \leq c_{1}\omega_{Q}\,\sum_{n=0}^{+\infty}\exp\left(
-\frac{2^{2n}d_{X}^{2}(x,y)}{t}\right)  \cdot\frac{2^{Q(n+1)}d_{X}^{Q}(x,y)}
{\vert B_{X}(x,\sqrt{t})\vert }\\[0.2cm]
&  =2^{Q}\,c_{1}\omega_{Q}\,\frac{t^{Q/2}}{\vert B_{X}(x,\sqrt
{t})\vert }\,\sum_{n=0}^{+\infty}\left(\frac{2^{n}d_{X}(x,y)}{\sqrt{t}%
}\right)  ^{Q}\exp\left(  -\frac{2^{2n}d_{X}^{2}(x,y)}{t}\right)  \\[0.2cm]
&  \quad\text{(by estimate \eqref{eq.stimaExpAlto}, with $\nu=Q$ and, e.g.,
$\delta=1/2$)}\\[0.2cm]
&  \quad\leq\frac{\alpha_{Q}\,t^{Q/2}}{\vert B_{X}(x,\sqrt{t})\vert}\,
\sum_{n=0}^{+\infty}\exp\left(-\frac{2^{2n}d_{X}^{2}(x,y)}{2t}\right)
=:(\bigstar\bigstar),
\end{align*}
for some constant $\alpha_{Q}$ depending on $Q$. On the other hand, since we
are assuming that $d_{X}(x,y)>\sqrt{t}$, for any $n\geq0$ we have
\begin{align*}
\exp\left(  -\frac{2^{2n}d_{X}^{2}(x,y)}{2t}\right)   &  =\exp\left(
-\frac{d_{X}^{2}(x,y)}{2t}\right)  \cdot\exp\left(  -\frac{d_{X}^{2}(x,y)}%
{2t}\cdot\left(  {2^{2n}-1}\right)  \right)  \\[0.2cm]
&  \leq\exp\left(  -\frac{d_{X}^{2}(x,y)}{2t}\right)  \cdot\exp\left(
-\frac{2^{2n}-1}{2}\right)  ,
\end{align*}
from which we derive that
\begin{align*}
(\bigstar\bigstar) &  \leq\frac{\alpha_{Q}\,s^{Q/2}}{\vert B_{X}%
(x,\sqrt{t})\vert }\,\exp\left(-\frac{d_{X}^{2}(x,y)}{2t}\right)
\cdot\sum_{n=0}^{+\infty}\exp\left(-\frac{2^{2n}-1}{2}\right)  \\[0.2cm]
&  =\frac{\alpha_{Q}^{\prime}\,t^{Q/2}}{\vert B_{X}(x,\sqrt{t})\vert }
\,\exp\left(-\frac{d_{X}^{2}(x,y)}{2t}\right).
\end{align*}
Finally, using this last estimate, we obtain
\begin{equation}\label{eq.StepFinaleCase1upper}%
\begin{split}
&  t^{-Q/2}\,\int_{\mathbb{R}^{p}}
\exp\left(-\frac{\|(x,0)^{-1}\ast(y,\eta)\|^{2}}{t}\right)\d\eta 
\leq\frac{\alpha_{Q}^{\prime}\,}{|B_{X}(x,\sqrt{t})|}\,\exp\left(
-\frac{d_{X}^{2}(x,y)}{2t}\right)
\end{split}
\end{equation}
which is precisely \eqref{eq.mainUpperestim} (with 
$\kappa=\alpha_{Q}^{\prime}$).

\medskip\textsc{Case II:} $d_{X}(x,y)\leq\sqrt{t}$. First of all, for every
non-negative integer $n$ we consider the set
\begin{equation}\label{eq.defisetAcase2upper}%
B_{n}:=\big\{  \eta\in\mathbb{R}^{p}:\,2^{n}\sqrt{t}\leq\|
(x,0)^{-1}\ast(y,\eta)\| <2^{n+1}\sqrt{t}\big\}
;
\end{equation}
moreover, we define
\begin{equation} \label{eq.defisetCAcase2}%
B:=\big\{  \eta\in\mathbb{R}^{p}:\,\|
(x,0)^{-1}\ast(y,\eta)\| <\sqrt{t}\big\}.
\end{equation}
Then we have:
\begin{align*}
&  \int_{\mathbb{R}^{p}}\exp\left(-\frac{\|(x,0)^{-1}\ast(y,\eta)\|^{2}}{t}\right)\d\eta
 = \int_{B}\big\{\ldots\big\}\,\d\eta
 +\sum_{n=0}^{+\infty}\int_{B_{n}}\big\{\ldots\big\}  \,\d\eta\\[0.2cm]
&  \quad\leq\left\vert B\right\vert +\sum_{n=0}^{+\infty}\exp\left(
-2^{2n}\right)  \cdot|B_{n}|\\[0.2cm]
&  \quad\leq
\big|\big\{\eta\in\mathbb{R}^{p}:\,(y,\eta)\in
B_{\mathcal{Z}}\big((x,0),\sqrt{t}\big)\big\}\big| \\[0.2cm]
&  \quad\quad\quad
+\sum_{n=0}^{+\infty}\exp\left(-2^{2n}\right)
\cdot
\big|\big\{\eta\in\mathbb{R}^{p}:\,(y,\eta)\in
B_{\mathcal{Z}}\big((x,0),2^{n+1}\sqrt{t}\big)\big\}\big| =:(\bigstar).
\end{align*}
Now, again by Theorem \ref{thm.globalSanchez} and \eqref{eq.measureBYCarnot},
for every $n\geq0$ we have
\begin{align*}
&  \big|\big\{\eta\in\mathbb{R}^{p}:\,(y,\eta)\in B_{\mathcal{Z}}
\big((x,0),2^{n}\sqrt{t}\big)\big\}\big\vert 
\leq c_{1}\,\frac{\big\vert B_{\mathcal{Z}}\big((x,0),2^{n}\sqrt{t}\big)\big\vert}
{\big\vert B_{X}\big(x,2^{n}\sqrt{t}\big)\big\vert}\\[0.2cm]
&  \qquad=c_{1}\omega_{Q}\,\frac{2^{nQ}t^{Q/2}}{|B_{X}(x,2^{n}\sqrt{t})|}
\leq c_{1}\omega_{Q}\,\frac{2^{nQ}\,t^{Q/2}}{| B_{X}(x,\sqrt{t})| }.
\end{align*}
As a consequence, we obtain
\begin{align*}
(\bigstar) &  \leq c_{1}\omega_{Q}\,\frac{t^{Q/2}}{| B_{X}(x,\sqrt{t})|}+
c_{1}\omega_{Q}\cdot\sum_{n=0}^{+\infty}\exp\left(
-2^{2n}\right)\,\frac{2^{(n+1)Q}t^{Q/2}}{| B_{X}(x,\sqrt{t})|}\\[0.2cm]
&  \quad=c_{1}\omega_{Q}\,\frac{t^{Q/2}}{| B_{X}(x,\sqrt{t})|}
\cdot\left(  1+\sum_{n=0}^{+\infty}\exp\left(-2^{2n}\right)\,2^{Q(n+1)}\right)  \\[0.2cm]
&  \quad=\frac{\beta_{Q}\,t^{Q/2}}{| B_{X}(x,\sqrt{t})|}=(\bigstar\bigstar).
\end{align*}
On the other hand, since we are assuming that $d_{X}(x,y)\leq\sqrt{t}$, we
have
$$
\exp\left(-\frac{d_{X}^{2}(x,y)}{2t}\right)  \geq e^{-1/2},
$$
from which we derive that
$$
(\bigstar\bigstar)\leq\beta_{Q}^{\prime}\,\frac{t^{Q/2}}{| B_{X}(x,\sqrt{t})|}
\cdot\exp\left(-\frac{d_{X}^{2}(x,y)}{2t}\right)  .
$$
Finally, using this last estimate, we obtain
\begin{equation} \label{eq.StepFinalecase2upper}
\begin{split}
&  t^{-Q/2}\,\int_{\mathbb{R}^{p}}
\exp\left( -\frac{\|(x,0)^{-1}\ast(y,\eta)\|^{2}}{t}\right)\d\eta
\leq\frac{\beta_{Q}^{\prime}}{| B_{X}(x,\sqrt{t})|}\cdot\exp\left(-\frac{d_{X}^{2}(x,y)}{2t}\right),
\end{split}
\end{equation}
and this is again \eqref{eq.mainUpperestim}. Gathering
\eqref{eq.StepFinaleCase1upper} and \eqref{eq.StepFinalecase2upper}, we
conclude that estimate \eqref{eq.mainUpperestim} holds for every
$x,y\in\mathbb{R}^{n}$ and every $t>0$ by choosing 
$$\kappa:=\max\{\alpha_{1}^{\prime},\beta_{Q}^{\prime}\}>1.$$ 
This ends the proof.
\end{proof}
In order to prove lower estimate of $\Gamma$, we need the following property.
\begin{lemma}
\label{lem.metricoReverseSanchez}
With the above notation and assumption, let
$b>a>0$ be fixed real numbers, and let $x,y\in\mathbb{R}^{n}$ satisfying%
\begin{equation}
d_{X}(x,y)<a. \label{eq.assumptiondxyleqa}%
\end{equation}
Then, for every $\xi\in\mathbb{R}^{p}$ there exists $\overline{\eta}%
=\overline{\eta}_{x,y,\xi}\in\mathbb{R}^{p}\setminus\{0\}$ such that
\begin{equation} \label{eq.inclusioncoronapalla}%
\begin{split}
 & \big\{\eta\in\mathbb{R}^{p}:a\leq d_{\mathcal{Z}}\big((x,\xi),(y,\eta)\big)<b\big\} 
  \supseteq\big\{\eta\in\mathbb{R}^{p}:(y,\eta)\in B_{\mathcal{Z}}
  \big((y,\overline{\eta}),\tfrac{1}{2}(b-a)\big)\big\}  . 
\end{split}
\end{equation}

\end{lemma}

\begin{proof}
Since $y\in B_{X}(x,a)$, if $\xi\in\mathbb{R}^{p}$ is arbitrarily fixed, by
\eqref{item.pjectionBY} there exists
\begin{equation} \label{eq.etazeroinsemeaperto}
\eta_{0}\in\big\{\eta\in\mathbb{R}^{p}:(y,\eta)\in B_{\mathcal{Z}}
\big((x,\xi),a\big)\big\}.%
\end{equation}
In particular, since the set in the right-hand side of
\eqref{eq.etazeroinsemeaperto} is open, we can assume that $\eta_{0}\neq0$.
\vspace*{0.07cm} We then consider the function $g:[1,+\infty)\rightarrow
\mathbb{R}$ defined as follows:
$$
g(\lambda):=d_{\mathcal{Z}}\big((x,\xi),(y,\delta_{\lambda}^{\ast}(\eta_{0}))\big),$$
where $\delta_{\lambda}^{\ast}(\eta)=(\lambda^{\tau_{1}}\eta_{1}%
,\ldots,\lambda^{\tau_{p}}\eta_{p})$ is as in \eqref{eq.expressionDlambda}.
Clearly, we have that $g$ is con\-ti\-nuo\-us
on the whole of $[1,+\infty)$; moreover, from
\eqref{eq.etazeroinsemeaperto} we infer that
\begin{equation} \label{eq.goneleqa}%
g(1)<a.
\end{equation}
We now claim that
\begin{equation} \label{eq.gdiverge}%
\lim_{\lambda\rightarrow+\infty}g(\lambda)=+\infty.
\end{equation}
To prove \eqref{eq.gdiverge} we first notice that, by triangle's inequality,
we have
\begin{equation} \label{eq.triangleg}
g(\lambda)\geq d_{\mathcal{Z}}\big((0,0),(y,\delta_{\lambda}^{\ast}(\eta_{0}))\big) 
 - d_{\mathcal{Z}}\big((0,0),(x,\xi)\big)
\qquad(\text{for all $\lambda\geq1$});
\end{equation}
moreover, since the vector fields
$Z_{1},\ldots,Z_{m}$ are $D_{\lambda}$-homogeneous of degree
$1$, by Proposition \ref{Prop properties distances}-(i) we deduce that
\begin{equation} \label{eq.tousedZinfinity}%
\begin{split}
&  d_{\mathcal{Z}}\big((0,0),(y,\delta_{\lambda}^{\ast}(\eta_{0}))\big)
=d_{\mathcal{Z}}\big((0,0),(\delta_{\lambda}(\delta_{1/\lambda}(y),
\delta_{\lambda}^{\ast}(\eta_{0}))\big)  \\[0.2cm]
&  \qquad(\text{setting $y_{\lambda}=\delta_{1/\lambda}(y)$})\\[0.2cm]
&  \qquad=d_{\mathcal{Z}}\big(  (0,0),D_{\lambda}(y_{\lambda},\eta_{0}))\big)  \\[0.2cm]
&  \qquad=\lambda\,d_{\mathcal{Z}}\big(  (0,0),(y_{\lambda},\eta_{0})\big)
.
\end{split}
\end{equation}
Since $y_{\lambda}=\delta_{1/\lambda}(y)\rightarrow0\in\mathbb{R}^{n}$ as
$\lambda\rightarrow+\infty$, and since $\eta_{0}\neq0$, we have%
$$
\lim_{\lambda\rightarrow+\infty}d_{\mathcal{Z}}
\big((0,0),(y_{\lambda},\eta_{0})\big)  =
 d_{\mathcal{Z}}\big((0,0),(0,\eta_{0})\big)  >0;
$$
as a consequence, taking the limit as $\lambda\rightarrow+\infty$ in
\eqref{eq.tousedZinfinity} we obtain
\begin{equation} \label{eq.dZdivergepolozero}%
\lim_{\lambda\rightarrow+\infty}d_{\mathcal{Z}}\big(  (0,0),(y,\delta
_{\lambda}^{\ast}(\eta_{0}))\big)  =+\infty.
\end{equation}
Gathering \eqref{eq.dZdivergepolozero} and \eqref{eq.triangleg}, we obtain
the claimed \eqref{eq.gdiverge}. \vspace*{0.07cm}

Next, using the continuity of $g$, together with \eqref{eq.goneleqa} and
\eqref{eq.gdiverge}, we infer the existence of a suitable 
$\overline{\lambda}\in(1,+\infty)$ such that
\begin{equation} \label{eq.choicelambda}%
g(\overline{\lambda})=d_{\mathcal{Z}}\big((x,\xi),
(y,\delta_{\overline{\lambda}}^{\ast}(\eta_{0}))\big)  =
\frac{b+a}{2}. 
\end{equation}
Setting $\overline{\eta}:=\delta_{\overline{\lambda}}^{\ast}(\eta_{0})$, we
prove \eqref{eq.inclusioncoronapalla} by showing the stronger inclusion
\begin{equation} \label{eq.inclusionballs}%
\big\{  z\in\mathbb{R}^{N}:\,a\leq d_{\mathcal{Z}}(  (x,\xi),z)
<b\big\}  \supseteq B_{\mathcal{Z}}\big(  (x,\overline{\eta}),\tfrac{1}%
{2}(b-a)\big)  . 
\end{equation}
To this end, let $u\in B_{\mathcal{Z}}\big(  (y,\overline{\eta}),\tfrac{1}%
{2}(b-a)\big)  $ be fixed. On the one hand, we have
$$
d_{\mathcal{Z}}\big(  (x,\xi),u\big)  \leq d_{\mathcal{Z}}\big(
(x,\xi),(y,\overline{\eta})\big)  +d_{\mathcal{Z}}\big(  (y,\overline{\eta
}),u\big)  <\frac{b+a}{2}+\frac{b-a}{2}=b;
$$
on the other hand, since we also have
$$
d_{\mathcal{Z}}\big((x,\xi),u\big)\geq d_{\mathcal{Z}}\big(
(x,\xi),(y,\overline{\eta})\big)  -d_{\mathcal{Z}}\big(  u,(y,\overline
{\eta})\big)  >\frac{b+a}{2}-\frac{b-a}{2}=a,
$$
we conclude that \eqref{eq.inclusionballs} holds. This ends the proof.
\end{proof}
We can now prove the estimate from below in \eqref{eq.aimGaussiantoprove}.
\begin{proposition} \label{prop.lowerestimateGamma} 
There exists a constant
$\vartheta>1$ such that 
\begin{equation} \label{eq.mainLowerestim}%
\begin{split}
& t^{-Q/2}\,\int_{\mathbb{R}^{p}}\exp
\left(-\frac{\|(x,0)^{-1}\ast(y,\eta)\|^{2}}{t}\right)\d\eta
\geq\frac{1}{\vartheta
\,|B_{X}(x,\sqrt{t})|}\,\exp\left(  -\frac{\vartheta\,d_{X}^{2}(x,y)}{t}\right),
\end{split}
\end{equation}
for every $x,y\in\mathbb{R}^{n}$ and
every $t>0$.
\end{proposition}

\begin{proof}
Let $x,y\in\mathbb{R}^{n}$ be arbitrarily fixed, and let $t>0$. \medskip

\textsc{Case I:} $d_{X}(x,y)>\sqrt{t}$. In this case, we consider the set
$$
A:=\big\{\eta\in\mathbb{R}^{p}:\,2d_{X}(x,y)\leq
\| (x,0)^{-1}\ast(y,\eta)\| <4d_{X}(x,y)\big\}  .
$$
By applying Lemma \ref{lem.metricoReverseSanchez} (with $a:=2d_{X}%
(x,y)>d_{X}(x,y)$ and $b:=2a$), one has
\begin{equation}  \label{inclusion A}%
A\supseteq\big\{  \eta\in\mathbb{R}^{p}:(y,\eta)\in B_{\mathcal{Z}}\big(
(y,\overline{\eta}),d_{X}(x,y)\big)\big\} 
\end{equation}
(for a suitable $\overline{\eta}=\overline{\eta}_{x,y}\in\mathbb{R}%
^{p}\setminus\{0\}$); as a consequence, we obtain
\begin{align*}
&  \int_{\mathbb{R}^{p}}\exp\left(  -\frac{\|(x,0)^{-1}\ast(y,\eta)\|^{2}}{t}\right)\d\eta
\geq\int_{A}\exp\left(
-\frac{\|(x,0)^{-1}\ast(y,\eta)\|^{2}}{t}\right)\d\eta\\[0.2cm]
&  \quad\text{(since $\|(x,0)^{-1}\ast(y,\eta)\|^{2}\leq 16\,d_{X}^{2}(x,y)$ for $\eta\in A$)}\\[0.2cm]
&  \quad\geq\exp\left(-\frac{16\,d_{X}^{2}(x,y)}{t}\right)  \cdot |A|\\[0.2cm]
&  \quad\geq\exp\left(  -\frac{16\,d_{X}^{2}(x,y)}{t}\right)  \cdot\big\vert
\big\{  \eta\in\mathbb{R}^{p}:(y,\eta)\in B_{\mathcal{Z}}\big(
(y,\overline{\eta}),d_{X}(x,y)\big)\big\}\big\vert =:(\bigstar).
\end{align*}
On the other hand, by using Theorem \ref{thm.globalSanchez} (with the choice
$(x,\xi)=(y,\overline{\eta})$) and \eqref{eq.measureBYCarnot}, we get
\begin{align*}
&  \big\vert
\big\{  \eta\in\mathbb{R}^{p}:(y,\eta)\in B_{\mathcal{Z}}\big(
(y,\overline{\eta}),d_{X}(x,y)\big)\big\}\big\vert \geq
c_{2}\,\frac{\big\vert B_{\mathcal{Z}}\big((y,\overline{\eta}),d_{X}(x,y)\big)\big\vert}
{\big\vert B_{X}\big(y,d_{X}(x,y)\big)\big\vert }\\[0.2cm]
&  \qquad=c_{2}\omega_{Q}\,\frac{d_{X}^{Q}(x,y)}{\big| B_{X}\big(y,d_{X}(x,y)\big)\big\vert}
\\[0.2cm]
&  \qquad\text{(since we are assuming that $d_{X}(x,y)>\sqrt{t}$)}\\[0.2cm]
&  \qquad>c_{2}\omega_{Q}\,\frac{t^{Q/2}}{\big| B_{X}\big(y,d_{X}(x,y)\big)\big\vert},
\end{align*}
from which we derive the estimate
\begin{align*}
(\bigstar) &  \geq c_{2}\omega_{Q}\,\frac{t^{Q/2}}{\big| B_{X}\big(y,d_{X}(x,y)\big)\big\vert}
\cdot\exp\left(  -\frac{16\,d_{X}^{2}(x,y)}%
{t}\right)  \\[0.2cm]
&  \text{(since $B_{X}\left(  y,d_{X}(x,y)\right)  \subseteq B_{X}\left(
x,2d_{X}(x,y)\right)  $)}\\[0.2cm]
&  \geq{c_{2}\omega_{Q}}\,\frac{t^{Q/2}}{\big| B_{X}\big(
x,2d_{X}(x,y)\big)  \big\vert }\cdot\exp\left(  -\frac{16\,d_{X}^{2}%
(x,y)}{t}\right)  =:(\bigstar\bigstar).
\end{align*}
We now observe that, since $X_{1},\ldots,X_{m}$ are $\delta_{\lambda}%
$-homogeneous of degree $1$, and since we are assuming that $d_{X}%
(x,y)>\sqrt{t}$, we can apply \eqref{eq.doublingglobalegeneral}, getting%
$$
\big| B_{X}\big(
x,2d_{X}(x,y)\big)  \big\vert \leq
\gamma_{2}\,\vert B_{X}(  x,\sqrt{t})\vert \cdot\left(
\frac{2d_{X}(x,y)}{\sqrt{t}}\right)^{q},
$$
where $q$ is as in \eqref{eq.defiqH1}. As a consequence, we deduce that
\begin{align*}
(\bigstar\bigstar) &  \geq\frac{c_{2}\omega_{Q}}{2^{q}\gamma_{2}}%
\,\frac{t^{Q/2}}{\vert B_{X}(  x,\sqrt{t})\vert}%
\cdot\left(  \frac{d_{X}(x,y)}{\sqrt{t}}\right)^{-q}\exp\left(
-\frac{16\,d_{X}^{2}(x,y)}{t}\right)  \\[0.2cm]
&  \text{(by estimate \eqref{eq.stimaExpBasso}, with $\nu=q$ and, e.g.,
$\theta=4$)}\\[0.2cm]
&  \geq\frac{t^{Q/2}}{\alpha_{q,Q}\,\vert B_{X}(  x,\sqrt{t})\vert}
\,\exp\left(  -\frac{20\,d_{X}^{2}(x,y)}{t}\right)  ,
\end{align*}
for some constant $\alpha_{q,Q}$ depending on $q,Q.$ Finally, by exploiting
this last e\-sti\-ma\-te, we obtain
\begin{equation} \label{eq.StepFinalecase1lower}%
\begin{split}
& t^{-Q/2}\,\int_{\mathbb{R}^{p}}\exp\left(-\frac{\|(x,0)^{-1}\ast(y,\eta)\|^{2}}{t}\right)\d\eta
\\[0.2cm]
&  \qquad\qquad\geq t^{-Q/2}\cdot\left[\frac{t^{Q/2}}{\alpha_{q,Q}\,
\vert B_{X}(x,\sqrt{t})\vert}
\,\exp\left(  -\frac{20\,d_{X}^{2}(x,y)}{t}\right)\right]  \\[0.2cm]
&  \qquad\qquad\text{(setting $\vartheta_{1}=\max\{\alpha_{q,Q},20\}$)} \\[0.2cm]
&  \qquad\qquad\geq\frac{1}{\vartheta_{1}\,\vert B_{X}(  x,\sqrt{t})\vert}
\,\exp\left(  -\frac{\vartheta_{1}\,d_{X}^{2}(x,y)}{t}\right),
\end{split}
\end{equation}
which exactly the desired \eqref{eq.mainLowerestim} (with $\vartheta
=\vartheta_{1}>1$). \medskip

\textsc{Case II:} $d_{X}(x,y)\leq\sqrt{t}$. The proof is similar to that of
\textsc{Case I}, letting now
$$
A= \big\{\eta\in\mathbb{R}^{p}:\,2\sqrt{t}\leq
\| (x,0)^{-1}\ast(y,\eta)\| < 4\sqrt{t}\big\}  .
$$
Applying Lemma \ref{lem.metricoReverseSanchez} (with $a:=2\sqrt{t}>d_{X}(x,y)$
and $b:=2a$), we get
$$
A\supseteq\big\{  \eta\in\mathbb{R}^{p}:(y,\eta)\in B_{\mathcal{Z}}(
(y,\overline{\eta}),\sqrt{t}) \big\}  \neq\varnothing
$$
(for a suitable $\overline{\eta}=\overline{\eta}_{x,y}\in\mathbb{R}^{p}\setminus\{0\}$); 
as a consequence, we obtain
\begin{align*}
&  \int_{\mathbb{R}^{p}}\exp\left(-\frac{\|(x,0)^{-1}\ast(y,\eta)\|^{2}}{t}\right)\d\eta
\geq \int_{A}\exp\left(-\frac{\|(x,0)^{-1}\ast(y,\eta)\|^{2}}{t}\right)\d\eta\\[0.2cm]
\quad &  \geq e^{-16}\cdot\left\vert A\right\vert \geq e^{-16}\cdot\big\vert
\big\{  \eta\in\mathbb{R}^{p}:(y,\eta)\in B_{\mathcal{Z}}(
(y,\overline{\eta}),\sqrt{t}) \big\}\big\vert =:(\bigstar).
\end{align*}
On the other hand, by using Theorem \ref{thm.globalSanchez} (with the choice
$(x,\xi)=(y,\overline{\eta})$) and \eqref{eq.measureBYCarnot}, we get
\begin{align*}
&  \big\vert
\big\{  \eta\in\mathbb{R}^{p}:(y,\eta)\in B_{\mathcal{Z}}(
(y,\overline{\eta}),\sqrt{t}) \big\}\big\vert \geq
c_{2}\,\frac{\big\vert B_{\mathcal{Z}}((y,\overline{\eta}),\sqrt{t})\big\vert}
{|B_{X}(y,\sqrt{t})\vert }\\[0.1cm]
&  \qquad=c_{2}\omega_{Q}\,\frac{t^{Q/2}}{\vert B_{X}(y,\sqrt{t})\vert },
\end{align*}
from which we derive the estimate (remind that we are assuming $d_{X}(x,y)\leq\sqrt{t}$)
\begin{align*}
(\bigstar) &  \geq\frac{c_{2}\omega_{Q}}{e^{16}}\,\frac{t^{Q/2}}
{\vert B_{X}(y,\sqrt{t})\vert }\\[0.2cm]
&  \text{\big(since $B_{X}(y,\sqrt{t})  \subseteq B_{X}(
x,d_{X}(x,y)+\sqrt{t})  \subseteq B_{X}(  x,2\sqrt{t})$\big)}\\[0.2cm]
&  \geq\frac{c_{2}\omega_{Q}}{e^{16}}\,\frac{t^{Q/2}}{
\vert B_{X}(x,2\sqrt{t})\vert }=:(\bigstar\bigstar).
\end{align*}
By \eqref{eq.doublingesplicita}, we have%
$$
\vert B_{X}(x,2\sqrt{t})\vert  \leq\gamma_{2}2^{q}\,|B_{X}(x,\sqrt
{t})|\qquad(\text{where $q$ is as in \eqref{eq.defiqH1}});
$$
as a consequence, we deduce that
$$
(\bigstar\bigstar)\geq\frac{c_{2}\omega_{Q}}{2^{q}e^{16}}\,\frac{t^{Q/2}%
}{|B_{X}(x,\sqrt
{t})|}\geq\frac{t^{Q/2}%
}{\beta_{q,Q}\,|B_{X}(x,\sqrt{t})|}%
\cdot\exp\left(  -\frac{d_{X}^{2}(x,y)}{t}\right)  ,
$$
for some constant $\beta_{q,Q}$ depending on $q,Q$.
Using this last estimate, we get
\begin{equation} \label{eq.StepFinalecase2lower}%
\begin{split}
&  t^{-Q/2}\,\int_{\mathbb{R}^{p}}\exp\left(-\frac{\|(x,0)^{-1}
 \ast(y,\eta)\|^{2}}{t}\right)\d\eta \\[0.2cm]
&  \qquad\qquad\geq t^{-Q/2}\cdot\left[
\frac{t^{Q/2}%
}{\beta_{q,Q}\,|B_{X}(x,\sqrt{t})|}%
\cdot\exp\left(  -\frac{d_{X}^{2}(x,y)}{t}\right)\right] \\[0.2cm]
&  \qquad\qquad\text{(setting $\vartheta_{2}:=\max\{\beta_{q,Q},1\}$%
)}\\[0.2cm]
&  \qquad\qquad\geq\frac{1}{\vartheta_{2}\,\vert B_{X}(x,\sqrt
{t})\vert }\,\exp\left(  -\frac{\vartheta_{2}\,d_{X}^{2}(x,y)}%
{t}\right)  ,
\end{split}
\end{equation}
and this is again the desired \eqref{eq.mainLowerestim} (this time with
$\vartheta=\vartheta_{2}\geq1$). Gathering \eqref{eq.StepFinalecase1lower} and
\eqref{eq.StepFinalecase2lower}, we conclude that estimate
\eqref{eq.mainLowerestim} holds for every $x,y\in\mathbb{R}^{n}$ and every
$t>0$ by choosing 
$$\vartheta:=\max\{\vartheta_{1},\vartheta_{2}\}>1.$$ 
This
ends the proof.
\end{proof}
Thanks to Propositions \ref{prop.upperestimateGamma} and
\ref{prop.lowerestimateGamma}, we can now \eqref{eq.GuassianGamma}
in Theorem \ref{thm.mainGaussianGamma}. \medskip

\begin{proof} [Proof of Theorem \ref{thm.mainGaussianGamma}-(i)] 
For every $x,y\in\mathbb{R}^{n}$ and every $t>0$, we set
$$
H(x,y,t):=t^{-Q/2}\,\int_{\mathbb{R}^{p}}\exp
\left(-\frac{\|(x,0)^{-1}\ast(y,\eta)\|^{2}}{t}\right)\d\eta.
$$
On account of \eqref{eq.estimGammaquasiGauss2}, we know that there exists a
constant $\mathbf{c}\geq1$, only depending on $\mathbb{G}$ and on
$\mathcal{Z}$ (which, in their turn, only depend on the set $X$), such that
\begin{equation}
\mathbf{c}^{{-1-Q/2}}\,H\left(  x,y,\mathbf{c}^{{-1}}t\right)  \leq
\gamma(t,x,y)\leq\mathbf{c}^{{-1-Q/2}}\,H\left(  x,y,\mathbf{c}t\right)
\label{eq.controlGammaH}%
\end{equation}
for every $x,y\in\mathbb{R}^{n}$ and every $t>0$. These bounds, together
with the preceding Propositions \ref{prop.upperestimateGamma} and
\ref{prop.lowerestimateGamma}, immediately give \eqref{eq.GuassianGamma}.
\end{proof}

\section{Estimates for the derivatives of $\Gamma$\label{sec.GaussianDer}}
The aim of this section is to establish (upper) Gaussian estimates for the
space derivatives along $X_{1},\ldots,X_{m}$ and for the `time derivatives' of
arbitrary order of $\gamma$, that is Theorem \ref{thm.mainGaussianGamma}-(ii).
To begin with, we state the following theorem proved in \cite{BBHeat}, which
provides integral re\-pre\-sen\-ta\-tions (analogous to formula
\eqref{defiGammaesplicitinaPRE1}) for any space/time derivative of $\gamma$.

\begin{theorem}[{See \cite[Theorem 3]{BBHeat}}]\label{Thm bbHeat}
Under the previous assumption,
and keeping the notation of Theorem \ref{ThmC}, for any nonnegative integers
$\alpha,h,k$ and any choice of indexes $i_{1},\ldots,i_{h},j_{1},\ldots,j_{k}$
in $\{1,\ldots,m\}$, we have the following representation formulas
\begin{align}
&  \left(\frac{\partial}{\partial t}\right)^{\alpha}
 X_{i_{1}}^{x}\cdots X_{i_{h}}^{x}\gamma(t,x,y)\label{derGammatx} \\[0.1cm]
 &  \quad=\int_{\mathbb{R}^{p}}\left(\left(
    \frac{\partial}{\partial t}\right)^{\alpha}
     {Z}_{i_{1}}\cdots{Z}_{i_{h}}\gamma_{\mathbb{G}}\right)
     \big(t,(y,0)^{-1}\ast(x,\eta)\big)\,\d\eta;	\nonumber \\[0.2cm]
&  \left(\frac{\partial}{\partial t}\right)^{\alpha}
X_{j_{1}}^{y}\cdots X_{j_{k}}^{y}\gamma(t,x,y)\label{derGammays} \\[0.1cm]
&  \quad = \int_{\mathbb{R}^{p}}
   \left(\left(  \frac{\partial}{\partial t}\right)^{\alpha}
   {Z}_{j_{1}}\cdots{Z}_{j_{k}}\gamma_{\mathbb{G}}\right)
   \big(t,(x,0)^{-1}\ast(y,\eta)\big)\,\d\eta; \nonumber\\[0.2cm]
&  \left(\frac{\partial}{\partial t}\right)^{\alpha}
X_{j_{1}}^{y}\cdots
X_{j_{k}}^{y}X_{i_{1}}^{x}\cdots X_{i_{h}}^{x}\gamma(t,x,y)\label{derGammatutte} \\[0.1cm]
   &  \qquad=\int_{\mathbb{R}^{p}}
   \left(\left(\frac{\partial}{\partial t}\right)^{\alpha}
   {Z}_{j_{1}}\cdots{Z}_{j_{k}}\big(({Z}_{i_{1}}\cdots{Z}_{i_{h}}\gamma_{\mathbb{G}})
   \circ\widetilde{\iota}\,\,\big)\right)\big(t,(x,0)^{-1}\ast(y,\eta)\big)\,\d\eta\,,\nonumber
\end{align}
holding true for every $(t,x)\neq(0,y)$ in $\mathbb{R}^{1+n}$. Here
$\widetilde{\iota}:\mathbb{R}^{1+N}\rightarrow\mathbb{R}^{1+N}$ is the map
defined by
$$
\widetilde{\iota}(t,u)=(t,u^{-1})
$$
and $u^{-1}$ is the inverse of $u$ in $\mathbb{G}=(\mathbb{R}^{N},\ast)$.
\end{theorem}
While the proof of our Gaussian estimates for the derivatives appearing in
\eqref{derGammays} and \eqref{derGammatx} is, by now, quite straightforward,
for the mixed case in \eqref{derGammatutte} it will require some extra work.
We start establishing the following proposition, which will be useful for the
case of mixed derivatives.

\begin{proposition} \label{Prop homogeneous BLU}
 With the above notation, for any nonnegative
integers $\alpha,h,k$ and any choice of indexes 
$i_{1},\ldots,i_{h},j_{1},\ldots,j_{k}\in \{1,\ldots,m\}$, there exists $c_1,c_2>0$ such that 
\begin{align*}
&  \left\vert 
   \left(\frac{\partial}{\partial t}\right)^{\alpha}
   {Z}_{j_{1}}\cdots{Z}_{j_{k}}\big( ({Z}_{i_{1}}\cdots{Z}_{i_{h}}\gamma_{\mathbb{G}})  
   \circ\widetilde{\iota}\,\,\big)(t,u)\right\vert 
   \leq c_1\,t^{-(Q+2\alpha+h+k)/2}\,\exp\left(-\frac{\|u\|^{2}}{c_2\,t}\right),
\end{align*}
for every $u\in\mathbb{G}$ and every $t>0$.
\end{proposition}
In turn, Proposition \ref{Prop homogeneous BLU} follows from two facts which
are stated separately in the next two lemmas, since they may be of independent interest.
\begin{lemma}
\label{Lemma der homog}Let $Y$ be a $1$-homogeneous 
\emph{(}but not necessarily left
invariant\emph{)} smooth vector field on $\mathbb{G}$. Then, it is possible to find 
another
$1$-homogeneous smooth vector field $\widetilde{Y}$ such that 
$$
Y(f\circ\iota)  = (\widetilde{Y}f)\circ\iota\qquad\text{for every $f\in C^\infty(\RN)$},
$$
where $\iota(u) = u^{-1}$ is the inversion map on $\mathbb{G}$.
\end{lemma}

\begin{proof}
First of all, let us write the dilations on $\mathbb{G}$ as:
$$
D_{\lambda}(u_{1},...,u_{N})  = 
(\lambda^{\alpha_{1}}u_{1},...,\lambda^{\alpha_{N}}u_{N})\qquad  
(\text{for any $\lambda>0$ and $u\in\mathbb{G}$}).
$$
We can write
$$
Y=\sum_{j=1}^{N}b_{j}(u)\,\frac{\de}{\partial u_{j}}
$$
where $b_{j}(u)$ is a $(\alpha_{j}-1)$-homogeneous
polynomial function. Moreover, 
using the structure of the inversion map on homogeneous groups (see, e.g.,
\cite[Corollary 1.3.16]{BLUlibro}), we know that the $k$-th com\-po\-nent of
$\iota(u) $ is a $\alpha_{k}$-homogeneous polynomial function.
Therefore
\begin{align*}
Y(f\circ\iota)(u)   &  =\sum_{j=1}^{N}b_{j}(u)\,
  \sum_{k=1}^{N}\frac{\partial f}{\partial u_{k}}\big(\iota(u)\big)  
  \frac{\partial \iota_{k}}{\partial u_{j}}(u) \\
 & =\sum_{k=1}^{N}\left(\sum_{j = 1}^N
   b_{j}(u)\frac{\partial \iota_{k}}{\partial u_{j}}(u)\right)
   \frac{\partial f}{\partial u_{k}}\big(\iota(u)\big) 
    \equiv\sum_{k=1}^{N}c_{k}(u) \frac{\partial f}{\partial u_{k}}\big(\iota(u)\big),
\end{align*}
and $c_{k}$ is a homogeneous polynomial function
 of degree 
 $$(\alpha_{j}-1)+(\alpha_{k}-\alpha_{j})  =\alpha_{k}-1.$$ 
 Next, we define $\widetilde{c}_{k}=c_{k}\circ\iota.$ 
 Since the dilations $D_\lambda$ are group automorphisms, we have \medskip
 
 (i)\,\,$D_{\lambda}\big(\iota(u)\big) = \iota\big(D_\lambda(u)\big)$; \medskip
 
 (ii)\,\,$\widetilde{c}_{k}\big(D_{\lambda}(u)\big) =
   c_{k}\big(D_{\lambda}\big(\iota(u)\big)\big)
  =\lambda^{\alpha_{k}-1}c_{k}\big(\iota(u)\big)
=\lambda^{\alpha_{k}-1}\widetilde{c}_{k}(u).$ \medskip

\noindent Hence, $\widetilde{c}_{k}$ is $(\alpha_{k}-1)$-homogeneous as well,
and
$$
Y(f\circ\iota)(u) =
\sum_{k=1}^{N}\widetilde{c}_{k}\big(\iota(u)\big)\frac{\partial f}{\partial u_{k}}\big(\iota(u)\big)
 \equiv (\widetilde{Y}f)\big(\iota(u)\big),
$$
where $\widetilde{Y}:=\sum_{k=1}^{N}\widetilde{c}_{k}(u)\,\partial_{u_{k}}$ 
is a $1$-homogeneous vector field (in view of (ii)).
\end{proof}
Next, let us prove the following:
\begin{proposition} \label{BLU estimates homogeneous}
Let $\alpha,r$ be nonnegative integers, and let
$Y_{1},...,Y_{r}$ be $1$-homogeneous \emph{(}but not necessarily left invariant\emph{)} 
smooth vector fields on $\G$.
Then, there exist constants $c_1,\,c_2 > 0$ such that, 
for every $u\in\mathbb{G}$ and every $t>0$, the following
Gaussian bound holds 
$$
\left\vert 
\left(\frac{\partial}{\partial t}\right)^{\alpha}
{Y}_{1}\cdots{Y}_{r}\gamma_{\mathbb{G}}(t,u)\right\vert \leq
c_1\,t^{-(Q/2+\alpha+r/2)}\,\exp\left(-\frac{\|u\|^{2}}{c_2\,t}\right).
$$
\end{proposition}

\begin{proof}
If $Y_{1},...,Y_{r}$ are $1$-homogeneous \emph{and left invariant} 
vector fields on $\G$, this result is proved by \cite[Theorem. 2.5]{BLUpaper}
(see also \eqref{eq.Gaussestder} in Theorem \ref{exTheoremB}). 
We are going to show that the result for left invariant $1$-homogeneous vector fields easily
implies our more general statement. \medskip

In fact, let $X_{1},...,X_{N}$ be 
the canonical basis of $\mathbb{G}$, i.e., $X_{i}$ is the
unique left invariant vector field on $\mathbb{G}$ such that
$X_i(0) = \partial_{u_{i}}$.
Up to possibly reordering the $X_i$'s,
we can assume that $X_{i}$ is $\alpha_{i}$-homogeneous, with 
$1=\alpha_{1}=...=\alpha_{m}<\alpha_{m+1}\leq\alpha_{m+2}....\leq\alpha_{N}=s$,
and
$s$ is the step of $\mathbb{G}$ (that is, $\Lie(\G)$ is nilpotent of step $s$). 
Then, for homogeneity
reasons, we have
$$
X_{i}=\partial_{u_{i}}+\sum_{\begin{subarray}{c}
k = 1 \\[0.05cm]
\alpha_{k}>\alpha_{i}
\end{subarray}}^N b_{ik}(u)\partial_{u_{k}}\qquad (\text{for $i=1,2,...,N$}),
$$
where $b_{ik}(u)$ is a $(\alpha_{k}-\alpha_{i})$-homogeneous polynomial
function. In particular, since $X_{N}=\partial_{u_{N}}$, we
can solve the above system in $\partial_{u_{1}},...,\partial_{u_{N}}$ using backward substitution,
thus writing%
\begin{equation} \label{eq.deucombXj}
\partial_{u_{i}}=X_{i}+\sum_{\begin{subarray}{c}
k = 1 \\[0.05cm]
\alpha_{k}>\alpha_{i}
\end{subarray}}^N c_{ik}(u)X_{k}\qquad (\text{for $i=1,2,...,N$}),
\end{equation}
where $c_{ik}(u)$ is a $(\alpha_{k}-\alpha_{i})$-homogeneous polynomial
function. 

Let now $Y$ be a $1$-homogeneous vector field. Owing to \eqref{eq.deucombXj},
we have
$$
Y=\sum_{i=1}^{N}\beta_{i}(u)\,\partial_{u_{i}}=
\sum_{i=1}^{N}\beta_{i}(u)\left(
X_{i}+\sum_{\begin{subarray}{c}
k = 1 \\[0.05cm]
\alpha_{k}>\alpha_{i}
\end{subarray}}^N c_{ik}(u)X_{k}\right) 
 \equiv\sum_{i=1}^{N}\gamma_{i}(u) X_{i},
$$
where $\gamma_{i}(u)$ is a $(\alpha_{i}-1)$-homogeneous
polynomial function. Notice that, since $X_{1},...,X_{m}$ are generators
of $\Lie(\G)$, every
$X_{i}$ with $i>m$ can be written as a linear combination (with constant
coefficients) of commutators of $X_{1},...,X_{m}$, of length $\alpha_{i}$.
Thus, since the Gaussian bound holds for left invariant vector fields
(see \eqref{eq.Gaussestder} in Theorem \ref{exTheoremB}), we obtain
\begin{align*}
\vert {Y}\gamma_{\mathbb{G}}(t,u)\vert  &  \leq
\sum_{i=1}^{N}\vert \gamma_{i}(u)\vert\cdot 
\vert X_{i}\gamma_{\mathbb{G}}(t,u)\vert
\leq \widehat{\mathbf{c}}\,\sum_{i=1}^{N}\vert \gamma_{i}(u)\vert\cdot
t^{-(Q+\alpha_{i})/2}\,\exp\left(-\frac{\|u\|^{2}}{\widehat{\mathbf{c}}\,t}\right) \\
&
 \leq \kappa\,\sum_{i=1}^{N}\|u\|^{\alpha_i-1}\cdot
t^{-(Q+\alpha_{i})/2}\,\exp\left(-\frac{\|u\|^{2}}{\widehat{\mathbf{c}}\,t}\right)
\\
& =
\kappa\,t^{-(Q+1)/2}\,\sum_{i=1}^{N}\left(\frac{\|u\|}{\sqrt{t}}\right)^{\alpha_i-1}
\!\!\!\!\cdot
\exp\left(-\frac{\|u\|^{2}}{\widehat{\mathbf{c}}\,t}\right)  \\
& \leq
c_1\,t^{-(Q+1)/2}\,\exp\left(-\frac{\|u\|^{2}}{c_2\,t}\right)
\end{align*}
where the last inequality follows from \eqref{eq.stimaExpAlto}. The general case
then follows by iteration.
\end{proof}
We are now ready to prove Proposition \ref{Prop homogeneous BLU}.
\begin{proof} [Proof of Proposition \ref{Prop homogeneous BLU}.]
By repeatedly applying Lemma \ref{Lemma der homog}, we can rewrite
\begin{align*}
\left(\frac{\partial}{\partial t}\right)^{\alpha}
   {Z}_{j_{1}}\cdots{Z}_{j_{k}}\big( ({Z}_{i_{1}}\cdots{Z}_{i_{h}}\gamma_{\mathbb{G}})  
   \circ\widetilde{\iota}\,\,\big) =
    \left\{ \left(\frac{\partial}{\partial t}\right)^{\alpha}
    \widetilde{{Z}}_{j_{1}}\cdots\widetilde{{Z}}_{j_{k}}
    \big(Z_{i_{1}}\cdots Z_{i_{h}}\gamma_{\mathbb{G}}\big)\right\}\circ
    \widetilde{\iota}
\end{align*}
with $\widetilde{\iota}(t,u)=(t,u^{-1})$. Here the $Z_{i}$'s are
$1$-homogeneous and left invariant, whereas the $\widetilde{Z}_{i}$'s are just
$1$-homogeneous.
Anyhow, we can apply Proposition
\ref{BLU estimates homogeneous} and get the desired result.
\end{proof}
With Proposition \ref{Prop homogeneous BLU} in hand, we can prove
the Gaussian estimates on the derivatives.
\begin{proof} [Proof of Theorem \ref{thm.mainGaussianGamma}-(ii)]
We
distinguish three different cases. \medskip

\textsc{Case 1.} $Y_{1},\ldots,Y_{r}=X_{i_{1}}^{x}\cdots X_{i_{r}}^{x}$. 
Then,
by \eqref{derGammatx}, \eqref{eq.Gaussestder} and Proposition 
\ref{prop.upperestimateGamma} we have
\begin{align*}
&  \left\vert 
  \left(\frac{\partial}{\partial t}\right)^{\alpha}
  X_{i_{1}}^{x}\cdots X_{i_{r}}^{x}\gamma(t,x,y)\right\vert \\[0.1cm]
&  \qquad 
   \leq \widehat{\mathbf{c}}\,
   t^{-(Q/2+\alpha+r/2)}\,\int_{\mathbb{R}^{p}}\exp\left(
-\frac{\|(y,0)^{-1}\ast(x,\eta)\|}{\widehat{\mathbf{c}}\,t}\right)\d\eta\\[0.1cm]
&  \qquad \leq c\,t^{-(\alpha+r/2)}\,\frac{1}
{\vert B_{X}(y,\sqrt{t})\vert}\,\exp\left(-\frac{d_{X}^{2}(x,y)}{C\,t}\right).
\end{align*}
The assertion then follows by Remark \ref{Remark equivalent gaussian}. \medskip

\textsc{Case 2.} $Y_{1},\ldots,Y_{r}=X_{j_{1}}^{y}\cdots X_{j_{r}}^{y}$. 
Then, by \eqref{derGammays}, \eqref{eq.Gaussestder} and 
Proposition \ref{prop.upperestimateGamma}, we have
\begin{align*}
&  \left\vert 
  \left(\frac{\partial}{\partial t}\right)^{\alpha}
  X_{j_{1}}^{y}\cdots X_{j_{r}}^{y}\gamma(t,x,y)\right\vert \\[0.1cm]
&  \qquad 
   \leq \widehat{\mathbf{c}}\,
   t^{-(Q/2+\alpha+r/2)}\,\int_{\mathbb{R}^{p}}\exp\left(
-\frac{\|(x,0)^{-1}\ast(y,\eta)\|}{\widehat{\mathbf{c}}\,t}\right)\d\eta\\[0.1cm]
&  \qquad \leq c\,t^{-(\alpha+r/2)}\,\frac{1}
{\vert B_{X}(x,\sqrt{t})\vert}\,\exp\left(-\frac{d_{X}^{2}(x,y)}%
{C\,t}\right).
\end{align*}

\textsc{Case 3.} $Y_{1},\ldots,Y_{r}=X_{j_{1}}^{y}\cdots X_{j_{k}}^{y}%
X_{i_{1}}^{x}\cdots X_{i_{h}}^{x}$ (with $k+h=r$). 
In this last case, by exploiting \eqref{derGammatutte}, 
Proposition \ref{Prop homogeneous BLU} and again Proposition \ref{prop.upperestimateGamma},
 we obtain
\begin{align*}
&  \left\vert 
    \left(\frac{\partial}{\partial t}\right)^{\alpha}
    X_{j_{1}}^{y}\cdots X_{j_{k}}^{y}X_{i_{1}}^{x}\cdots X_{i_{h}}^{x}
    \gamma(t,x,y)\right\vert \\[0.1cm]
&  \qquad \leq c_1\,t^{-(Q/2+\alpha+r/2)}\,
    \int_{\mathbb{R}^{p}}\exp\left(-\frac{\|(x,0)^{-1}\ast(y,\eta)\|^{2}}{c_2\,t}\right)\d\eta\\[0.1cm]
&  \qquad \leq c\,t^{-(\alpha+r/2)}\,\frac{1}{\vert B_{X}(x,\sqrt{t})\vert}\,
\exp\left(-\frac{d_{X}^{2}(x,y)}{C\,t}\right).
\end{align*}
This ends the proof.
\end{proof}

\section{An application to the Cauchy problem for $\mathcal{H}$} \label{sec.Cauchy}
As anticipated in the Introduction, in this section we exploit the global
Gaussian bounds of $\Gamma$ to study the unique solvability of the Cauchy
problem for $\mathcal{H}$. More precisely, we extend the result proved in
\cite[Thm. 4.1]{BBHeat}, where the Cauchy problem is studied for bounded
continuous initial data, to possibly unbounded continuous initial data, fastly
growing at infinity.

As for the proof of Gaussian estimates, we will make Assumptions
\ref{Assumptions} and\ will also assume (H3) (see Section \ref{sec.preliminari}). 
As noted before, condition (H3) amounts to assuming
that we are \emph{not} in a Carnot group (see also Remark
\ref{Remark Carnot o no} after the proof of our result for some explanation on
this point). \medskip

We start with the following
\begin{definition}
Let $S_{\tau}:=(0,\tau)\times\mathbb{R}^{n}$, for some fixed $\tau
\in(0,+\infty]$. Given any function $f\in C(\mathbb{R}^{n})$,
we say that $u:S_{\tau}\rightarrow\mathbb{R}$ is a \emph{classical solution}
of the Cauchy problem
\begin{equation} \label{eq.mainCauchyPbHt}%
\begin{cases}
\mathcal{H}u=0 & \text{in $S_{\tau}$},\\
u(0,x)=f(x) & \text{for $x\in\mathbb{R}^{n}$}%
\end{cases}
\end{equation}
if it satisfies the following properties:
\begin{enumerate}
\item $u\in C^{2}(S_{\tau})$ and $\mathcal{H}u=0$ pointwise on $S_{\tau}$;
\item $\lim_{t\rightarrow0^{+}}u(t,x)=f(x)$ for every fixed $x\in\mathbb{R}^{n}$.
\end{enumerate}
\end{definition}
Using the upper Gaussian estimates of $\Gamma$, we are able to prove that
\eqref{eq.mainCauchyPbHt}) admits (at least) one classical solution when the
initial datum $f$ grows at most exponentially. In what follows, we set
$$
\rho_{X}(x):=d_{X}(0,x)\qquad(x\in\mathbb{R}^{n}).
$$
\begin{theorem} \label{thm.mainExistenceCauchy}
There exists $T>0$ such that, if $f\in C(\mathbb{R}^{n})$ satisfies
the growth condition
\begin{equation} \label{eq.growthf}%
\int_{\mathbb{R}^{n}}\vert f(y)\vert\,\exp\big(-\mu\rho_{X}^{2}(y)\big)\,\d y<+\infty
\end{equation}
for some constant $\mu>0$, then the function
\begin{equation} \label{eq.defiusolution}%
u(t,x):=\int_{\mathbb{R}^{n}}\Gamma(t,x;0,y)\,f(y)\,\d y=
\int_{\mathbb{R}^{n}}\gamma(t,x,y)\,f(y)\,\d y
\end{equation}
is a classical solution of \eqref{eq.mainCauchyPbHt} on the strip $S_{T/\mu}$.

Furthermore, it is possible to find constants $\tau,\delta>0$ 
\emph{(}depending on $\mu$\emph{)} such that
\begin{equation} \label{eq.growthuunique}%
\int_{S_{\tau}}\exp\left(-\delta\rho_{X}^{2}(x)\right)\,|u(t,x)|\,\d t\,\d x<+\infty. 
\end{equation}
Finally, if $u_{1},u_{2}$ are two classical solutions of
\eqref{eq.mainCauchyPbHt} with the same continuous initial datum $f$, and
if $u_{1},u_{2}$
satisfy condition \eqref{eq.growthuunique} in two strips 
$S_{\tau_{1}},S_{\tau_{2}}$, respectively, then 
$$\text{$u_{1}\equiv u_{2}$ in $S_{\tau}$, \qquad for $\tau=\min\{\tau_{1},\tau_{2}\}$}.$$
\end{theorem}
Before proving Theorem \ref{thm.mainExistenceCauchy} we establish, for a
future reference, the following easy lemma.
\begin{lemma} \label{lem.integrabilityexp}
For every fixed $\theta>0$, we have
\begin{equation} \label{eq.integrabphiexp}%
\phi(y):=\exp\left(-\theta\rho_{X}^{2}(y)\right)  \in L^{1}(\mathbb{R}^{n}).
\end{equation}
\end{lemma}
\begin{proof}
Since $\phi$ is bounded, it suffices to show that $\phi$ is integrable at
infinity. To this end, if $\sigma_{1},\ldots,\sigma_{n}$ are as in
\eqref{eq.defidela}, we consider the \emph{homogeneous norm}
$$
\mathcal{N}(y):=\sum_{j=1}^{n}|y_{j}|^{1/\sigma_{j}}\qquad(y\in\mathbb{R}^{n}),
$$
and we prove that $\phi$ is integrable on the set $\mathcal{O}:=\{\mathcal{N}%
\geq1\}$. Now, using Lemma \ref{lem.stimeExp}, and taking into account that
both $\mathcal{N}$ and $\rho_{X}\ $are $\delta_{\lambda}$-homogeneous of
degree $1$, we have%
\begin{align*}
&  \int_{\mathcal{O}}\phi(y)\,\d y\leq 
  c\,\int_{\mathcal{O}}\frac{1}{\rho_{X}^{2q}(y)}\,\d y=
  c\,\sum_{k=0}^{+\infty}\int_{\{2^{k}\leq\mathcal{N}<2^{k+1}\}}\frac{1}{\rho_{X}^{2q}(y)}\,\d y\\[0.2cm]
  &  \qquad(\text{performing the change of variable $y=\delta_{2^{k}}(u)$})\\[0.2cm]
  &  \qquad=c\,\left(
  \int_{\{1\leq\mathcal{N}<2\}}\frac{1}{\rho_{X}^{2q}(u)}\,\d u\right)  
  \cdot\sum_{k=0}^{+\infty}\frac{1}{2^{qk}}<+\infty.
\end{align*}
This ends the proof.
\end{proof}
We are now ready to prove Theorem \ref{thm.mainExistenceCauchy}.
\begin{proof} [Proof of Theorem \ref{thm.mainExistenceCauchy}]
Assume that $f\in C(\R^n)$ satisfies \eqref{eq.growthf}, and let $u$ be as in
\eqref{eq.defiusolution}. \medskip

\textsc{Step I.} Let us show that $u$ is well defined and solves 
 \eqref{eq.mainCauchyPbHt}
 in some $S_{T}$. To this end, let $R > 0$ be arbitrarily fixed,
 and let $\phi_{R}\in C_{0}^{0}(\mathbb{R}^{n})$ satisfy the following properties
\begin{itemize}
 \item $\phi_{R} \equiv 1$ on $\{\rho_{X}<R\}$;
 \item $\phi_{R} \equiv 0$ on $\{\rho_{X}>2R\}$;
 \item $0\leq \phi_R\leq 1$ on $\R^n$.
\end{itemize}
 Then, we can write
\begin{align*}
 u(t,x) =
 \int_{\mathbb{R}^{n}}\gamma(t,x,y)\,f(y)\phi_{R}(y)\,\d y
 + 
 \int_{\mathbb{R}^{n}}\gamma(t,x,y)\,f(y)\big(1-\phi_{R}(y)\big)\,\d y
 \equiv u_{1}(t,x) + u_{2}(t,x).
\end{align*}
 Since $f\phi_{R}$ is bounded continuous, by \cite[Theorem 4.1]{BBHeat} we know
 that $u_{1}$ is well defined for every $t>0$, and it solves 
 \eqref{eq.mainCauchyPbHt}
 with initial datum $f\phi_{R}$ on the whole of on $(0,+\infty)\times\R^n$. In particular,
 since $\phi_R\equiv 1$ on the set
  $\{\rho_X < R\}$, for every $x\in\R^n$ with $\rho_X(x) < R$
 we have
 $$\lim_{t\rightarrow 0^{+}} u_{1}(t,x) = (f\phi_R)(x) = f(x).$$
 We now prove that there exists a suitable $T > 0$, \emph{independent of the chosen $R$},
 such that the following facts hold on the
 bounded stripe $S_{T/\mu,R} := (0,T/\mu)\times\{\rho_X<R\}$:
 $$\text{(i)\,\,$u_{2}$ is well defined; \quad (ii)\,\,$u_2$ it solves 
 the equation $\mathcal{H}u=0$;\quad
 (iii)\,\,$u_{2}(t,x)\rightarrow 0$ as $t\rightarrow0^{+}$}.$$
 As for (i)-(ii) we observe that, by the Gaussian estimate 
 \eqref{eq.GuassianGamma}, we have
\begin{equation} \label{eq.estimutxprima}
\begin{split}
& \vert u_{2}(t,x)\vert 
\leq\frac{\varrho}{\vert B(x,\sqrt{t})\vert}
\int_{\{\rho_{X}(y)>2R\}}\exp\left(-\frac{d_X^2(x,y)}{\varrho t}\right)\,
 \vert f(y)\vert \,\d y \\[0.1cm]
& \quad = \frac{\varrho}{\vert B(x,\sqrt{t})\vert}
  \int_{\{\rho_{X}(y) > 2R\}}\exp\left(-\frac{d_X^2(x,y)}{\varrho t}
  +\mu\rho_{X}^{2}(y)\right)\,\vert f(y)\vert\,\exp
  \big(-\mu\rho^{2}_{X}(y)\big)\,\d y.
\end{split}
\end{equation}
On the other hand, for every $x,y\in\R^n$ satisfying
$\rho_X(x) < R$ and $\rho_X(y) > 2R$, one has
$$
d_X(x,y) \geq\rho_{X}(y)  -\rho_{X}(x)  
 \geq \frac{\rho_{X}(y)}{2};$$
 as a consequence, we get
\begin{equation} \label{eq.expleqdausare}
\exp\left(-\frac{d_X^{2}(x,y)}{\varrho t}
  +\mu\rho_{X}^{2}(y)\right)
  \leq 
  \exp\left(-\rho_X^2(y)\,\Big(\frac{1}{4\varrho t}-\mu\Big)\right)\leq 1,
\end{equation}
as soon as $x\in\{\rho_X < R\}$ and
$\frac{1}{4\varrho t}-\mu > 0$, that is
(setting $T_1 := 1/(4\varrho)$)
$$t < \frac{T_1}{\mu}.$$ 
Gathering together all these facts, for fixed $(t,x)\in S_{T_1/\mu, R}$ we obtain
$$
\vert u_{2}(t,x)\vert \leq c_{t,x}\int_{\mathbb{R}^{n}}
 \vert f(y)\vert\,\exp\big(-\mu\rho_{X}^{2}(y)\big)\,\d y < +\infty.
$$
Now, using the Gaussian estimates \eqref{Gassian_derivatives} for the
derivatives of $\gamma$, and arguing exactly as above, 
one can easily prove that
$u_2\in C^2(S_{T/\mu, R})$ and $\mathcal{H}u_{2}=0$
on $S_{T/\mu, R}$, where 
\begin{equation} \label{eq.choiceTCauchy} 
 T := \min\bigg\{T_1,\frac{1}{4C}\bigg\}\quad
\text{and $C$ is as in \eqref{Gassian_derivatives}}.
\end{equation}
Next, we show that for $t\rightarrow 0^+$ we have $u_{2}(t,x)\rightarrow0$
if $x\in\R^n$ satisfies $\rho_X(x) < R$. To this end we
first observe that, by \eqref{eq.NSWmodificata}, for every $t>0$ and $x\in\mathbb{R}^{n}$ we have
\begin{equation} \label{lower ball}%
\vert B_{X}(x,\sqrt{t})\vert \geq\gamma_{1}\,\sum_{h=n}^{q}
  f_{h}(y)\,t^{h/2}\geq\gamma_{1}\,f_{q}\,t^{q/2}=\kappa_{q}\,t^{q/2},
\end{equation}
with $\kappa_q := \gamma_1\,f_q$ 
(remind that $f_{n},\ldots,f_{q}\geq0$ and $f_{q}$ is a positive constant).
As a consequence, by combining \eqref{lower ball},
\eqref{eq.estimutxprima} and \eqref{eq.expleqdausare},
for every $(t,x)\in S_{T/\mu, R}$ we obtain the estimate
$$
|u_{2}(t,x)| \leq c\int_{\{\rho_{X}(y) > 2R\}}
\frac{e^{-\rho_{X}^{2}(y)\big(\frac{1}{4\varrho t}-\mu\big)}}
{t^{q/2}}\cdot|f(y)|\,\exp\big(-\mu\rho^2_{X}(y)\big)\,\d y.
$$
We are going to show that, by Lebesgue's theorem, the last integral goes to
zero as $t\rightarrow0^{+}$. On the one hand,
for every fixed $y\in\R^n$ with $\rho_X(y) > 2R$, we have
$$
\lim_{t\to 0^+}
\bigg(\frac{e^{-\rho_{X}^2(y)\big(\frac{1}{4\varrho t}-\mu\big)}}
 {t^{q/2}}\cdot \vert f(y)\vert\,e^{-\mu\rho_{X}^2(y)}\bigg)
 \leq c_{y}\cdot \lim_{t\to 0^+}
 \frac{e^{-4R^{2}\big(\frac{1}{4\varrho t}-\mu\big)}}{t^{q/2}} = 0.
$$
On the other hand, for every $t > 0$ and every $y\in\R^n$ satisfying
$\rho_X(y) > 2R$, one has
\begin{align*}
& \frac{e^{-\rho_{X}^{2}(y)\big(\frac{1}{4\varrho t}-\mu\big)}}
{t^{q/2}}\cdot \vert f(y)\vert\,e^{-\mu\rho_{X}^2(y)}   \leq
 \frac{e^{-4R^{2}\big(\frac{1}{4\varrho t}-\mu\big)}}{t^{q/2}}
 \cdot \vert f(y)\vert\,e^{-\mu\rho_{X}^2(y)} \\
& \qquad
  = \sup_{t > 0}\bigg(\frac{e^{-4R^{2}\big(\frac{1}{4\varrho t}-\mu\big)}}{t^{q/2}}\bigg)
    \cdot \vert f(y)\vert\,e^{-\mu\rho_{X}^2(y)}
    \equiv c\,\vert f(y)\vert\,e^{-\mu\rho_{X}^2(y)}\in L^{1}(\mathbb{R}^{n}),
\end{align*}
 and thus, by Lebesgue's theorem, we conclude that
 $$\lim_{t\to 0^+}u_{2}(t,x) = 0 \qquad\text{for every $x\in\R^n$ with $\rho_X(x) < R$}.$$
 Summing up, we have proved that $u_2$ satisfies (i)-to-(iii)
 on $S_{T/\mu, R}$, as desired. \vspace*{0.05cm}
 
 Finally, due to the arbitrariness of $R > 0$, 
 we then conclude that $u$ is a classical solution of
 problem \eqref{eq.mainCauchyPbHt} on the stripe $S_{T/\mu}$ 
 (with $T$ as in \eqref{eq.choiceTCauchy}). \medskip
 
 \textsc{Step II.} Let us show that $u$ satisfies a bound 
 \eqref{eq.growthuunique} for some $\delta,\tau > 0$. 
 Since $u$ is a continuous function on the stripe $S_{T/\mu}$, the integral
 $$
  \int_{0}^{\tau}\int_{\{\rho_X(x)\leq R\}}\vert u(t,x)\vert\,
  \exp\big(-\delta\rho_{X}^{2}(x)\big)\,\d t\,\d x
 $$
is finite for every choice of $\delta,R > 0$ and every $0<\tau<T/\mu$. 
So, it is enough to show that there exist suitable
 $\delta \in (0,+\infty)$ and $0<\tau<T/\mu$ such that
$$ 
 \int_{0}^{\tau}\int_{\{\rho_X(x)> 1\}}\vert u(t,x)\vert\,
  \exp\big(-\delta\rho_{X}^{2}(x)\big)\,\d t\,\d x < +\infty.
 $$
By the very definition of $u$ in \eqref{eq.defiusolution}, we have
\begin{align*}
&  \int_{0}^{\tau}\int_{\{\rho_X(x)> 1\}}\vert u(t,x)\vert\,
  \exp\big(-\delta\rho_{X}^{2}(x)\big)\,\d t\,\d x \\[0.1cm]
& \qquad \leq\int_{0}^{\tau}\int_{\{\rho_X(x)> 1\}}\left(
  \int_{\mathbb{R}^{n}}\gamma(t,x,y)\,\vert f(y)\vert \,\d y\right)
  \exp\big(-\delta\rho_{X}^{2}(x)\big)\,\d t\,\d x.
\end{align*}
We then split the space integral as follows
\begin{align*}
& \int_{\{\rho_X(x)> 1\}}\left(
  \int_{\mathbb{R}^{n}}\gamma(t,x,y)\,\vert f(y)\vert \,\d y\right)
  \exp\big(-\delta\rho_{X}^{2}(x)\big)\,\d x \\[0.1cm]
&  \qquad 
   = \int_{\{\rho_X(x)> 1\}}\left(
  \int_{\{\rho_X(y)\geq 2\rho_X(x)\}}\gamma(t,x,y)\,\vert f(y)\vert \,\d y\right)
  \exp\big(-\delta\rho_{X}^{2}(x)\big)\,\d x \\
&  \qquad\quad + \int_{\{\rho_X(x)> 1\}}\left(
  \int_{\{\rho_X(y)< 2\rho_X(x)\}}\gamma(t,x,y)\,\vert f(y)\vert \,\d y\right)
  \exp\big(-\delta\rho_{X}^{2}(x)\big)\,\d x \\[0.1cm]
&  \qquad \equiv A(t)+B(t) .
\end{align*}
 As for $A(t)$, by combining the Gaussian estimate \eqref{eq.GuassianGamma} with \eqref{lower ball},
 we get
\begin{align*}
 A(t)  & \leq \frac{c\varrho}{t^{q/2}}\,\int_{\{\rho_{X}(x) > 1\}}
 \left(\int_{\{\rho_X(y)\geq 2\rho_X(x)\}}\!\!\!\!\!\!
  e^{-\frac{d_X^2(x,y)}{\varrho t}+\mu\rho_{X}^{2}(y)}\cdot 
  e^{-\mu\rho_{X}^{2}(y)}\,\vert f(y)\vert \,\d y \right)\cdot
   e^{-\delta\rho_{X}^{2}(x)}\,\d x;
\end{align*}
 moreover, using the fact that $d_X(x,y)\geq \rho_X(y)-\rho_X(x)$ for every $x,y\in\R^n$, 
 one has
\begin{align*}
 \exp\bigg(-\frac{d_X^{2}(x,y)}{\varrho t}+\mu\rho_{X}^2(y)\bigg) 
 \leq \exp\bigg(\frac{\rho_X^2(x)}{\varrho t}\bigg) 
 \cdot\exp\bigg(-\rho_X^2(y)\Big(\frac{1}{2\varrho t}-\mu\Big)\bigg)
  = (\bigstar).
\end{align*}
 As a consequence, since in $A(t)$ we have
 $\rho_X(y)\geq 2\rho_X(x)$ and $\rho_X(x) > 1$, we obtain
 \begin{align*}
  (\bigstar)
  & \leq \exp\bigg(\frac{\rho_X^2(x)}{\varrho t}-4\rho_X^2(x)
   \Big(\frac{1}{2\varrho t}-\mu\Big)\bigg)
   = \exp\bigg(-\rho_X^2(x)\Big(\frac{1}{\varrho t}-4\mu\Big)\bigg) \\[0.1cm]
   & \leq \exp\bigg(-\Big(\frac{1}{\varrho t}-4\mu\Big)\bigg),
 \end{align*}
 provided that $t\in (0,T/\mu)$, see \eqref{eq.choiceTCauchy}. Using this last estimate,
 we get
 \begin{align*}
 A(t) &  \leq\frac{c\varrho}{t^{q/2}}\,e^{-\big(\frac{1}{\varrho t}-4\mu\big)}\,
 \left(\int_{\mathbb{R}^{n}} |f(y)|\,\exp\big(-\mu\rho_{X}^2(y)\big)\,\d y\right)
  \left(\int_{\mathbb{R}^{n}}\exp\big(-\delta\rho_{X}^{2}(x)\big)\,\d x\right)  
  \\[0.2cm]
&  =\frac{c'}{t^{q/2}}e^{-\big(\frac{1}{\varrho t}-4\mu\big)}%
\end{align*}
where we have exploited Lemma \ref{lem.integrabilityexp}. From this, we finally obtain
$$
\int_{0}^{\tau}A(t)\,\d t
 \leq \int_{0}^{\tau}\frac{c_{1}}{t^{q/2}}e^{-\big(\frac{1}{\varrho t}-4\mu\big)}\,\d t<
 +\infty, \quad \text{for any $\tau\in (0,T/\mu)$ and any $\delta > 0$. }
$$
As for $B(t)$, since $\rho_{X}(y) < 2\rho_X(x)$, we have
\begin{align*}
B(t)  &  
 = \int_{\{\rho_X(x)> 1\}}\left(
  \int_{\{\rho_X(y)< 2\rho_X(x)\}}\gamma(t,x,y)\,e^{\mu\rho_X^2(y)}\cdot
   \vert f(y)\vert\,e^{-\mu\rho_X^2(y)} \,\d y\right)
  e^{-\delta\rho_{X}^{2}(x)}\,\d x \\[0.2cm]
  & \leq 
  \int_{\R^n}|f(y)|\,\exp\big(-\mu\rho_X^2(y)\big)\left(
  \int_{\R^n}\gamma(t,x,y)\,e^{(4\mu-\delta)\rho_X^2(x)}\,\d x\right)\d y.
\end{align*}
Thus, if we choose $\delta\geq4\mu$, from Theorem \ref{exTheoremB}-(iv) and (vi) we obtain
$$
\int_{\mathbb{R}^{n}}\gamma(t,x,y)\,e^{-(\delta-4\mu)\rho_X^2(x)}\,\d x
 \leq \int_{\mathbb{R}^{n}}\gamma(t,x,y)\,\d x=1.
 $$
 As a consequence, we get
 $$
 B(t) \leq\int_{\mathbb{R}^{n}}\vert f(y)\vert\,e^{-\mu\rho_X^2(y)} \,\d y =: c <+\infty,
 $$
 from which we derive that
 $$\int_{0}^{\tau}B(t)\,\d t<+\infty\qquad\text{for any $\tau \in (0,T/\mu)$
  and any $\delta\geq 4\mu$}.$$
 Summing up, we conclude that $u$ satisfies
 \eqref{eq.growthuunique} for every $\delta\geq 4\mu$ and every $\tau\in (0,T/\mu)$.
 \medskip
 
 \textsc{Step III.} Let us prove the uniqueness result. By linearity, it is
enough to show that if for some $\tau>0$ the function $u\in C^{2}(S_{\tau})$
is a classical solution of
\begin{equation} \label{eq.mainPbHomog}%
 \begin{cases} 
 \mathcal{H}u=0 & \text{in $S_{\tau}$},\\
 u(0,x)=0 & \text{for $x\in\mathbb{R}^{n}$}%
 \end{cases}
\end{equation}
and satisfies \eqref{eq.growthuunique}, then $u\equiv0$ on $S_{\tau}$.
Denoting again by $\pi_{n}$ the projection of $\mathbb{R}^{N}$ onto
$\mathbb{R}^{n}$, we set
$$
\widehat{u}:\widehat{S}_{\tau}:=(0,\tau)\times\mathbb{R}^{N}\rightarrow
\mathbb{R},\qquad\widehat{u}(t,z):=u\big(t,\pi_{n}(z)\big)  .
$$
Obviously, $\widehat{u}\in C^{2}(\widehat{S}_{\tau})$; moreover, since $u$
solves \eqref{eq.mainPbHomog} and $\mathcal{H}_{\mathbb{G}}=
\sum_{j=1}^{m}Z_{j}^{2}-\partial_{t}$ is a lifting of $\mathcal{H}$ (see
\eqref{eq.liftinHHG}), it is easy to check that $\widehat{u}$ is a classical
solution of
\begin{equation} \label{eq.mainPbCauchyhomG}%
\begin{cases}
\mathcal{H}_{\mathbb{G}}\widehat{u}=0 & \text{in $\widehat{S}_{\tau}$},\\
\widehat{u}(0,z)=0 & \text{for $z\in\mathbb{R}^{N}.$}%
\end{cases}
\end{equation}
 We claim that there exists $\widehat{\delta}>0$ such that
\begin{equation} \label{eq.growthvunique}%
\int_{\widehat{S}_{\tau}}\exp\big(-\widehat{\delta}\,\|z\|^{2}\big)
 \,|\widehat{u}(t,z)|\,\d t\,\d z<+\infty.
\end{equation}
Once this is proved, by \cite[Theorem 6.5]{BLUpaper} we derive that
$\widehat{u}\equiv0$ on $\widehat{S}_{\tau}$, and thus $u\equiv0$ on $S_{\tau}.$
\medskip

 To prove \eqref{eq.growthvunique}, let $\widehat{\nu}>0$ to be fixed
in a moment. By using Proposition \ref{prop.upperestimateGamma} (with
$x = 0$ and $t=\widehat{\delta}^{-1}>0$), we obtain the following computation
\begin{equation}\label{eq.mainestimUnique}%
\begin{split}
&  \int_{\widehat{S}_{\tau}}\exp\big(-\widehat{\delta}\,
 \|z\|^{2}\big)\,|\widehat{u}(t,z)|\,\d t\,\d z=
 \int_{\widehat{S}_{\tau}}
 \exp\big(-\widehat{\delta}\,\|(x,\xi)\|^2\big)\,
 |\widehat{u}(t,(x,\xi))|\,\d t\,\d x\,\d\xi\\[0.2cm]
&  \qquad=\int_{S_{\tau}}\left(\int_{\mathbb{R}^{p}}\exp\big(-\widehat{\delta}\,\|
 (x,\xi)\|^{2}\big)\,\d\xi\right) |u(t,x)|\,\d t\,\d x\\[0.2cm]
&  \qquad\leq\frac{\kappa}{\widehat{\delta}^{Q/2}\,|
 B_{X}(0,\widehat{\delta}^{-1/2})| }\,\int_{S_{\tau}}\exp\left(
-\frac{\widehat{\delta}\,\rho_{X}^{2}(x)}{\kappa}\right)|u(t,x)|\,\d t\,\d x,
\end{split}
\end{equation}
for a suitable constant $\kappa>1$ . As a
consequence, if we choose $\widehat{\delta}:=\delta\cdot\kappa$, with $\delta$ as
in \eqref{eq.growthuunique}, from \eqref{eq.mainestimUnique} we immediately
deduce the claimed \eqref{eq.growthvunique}. This ends the proof.
\end{proof}
\begin{remark}

\label{Remark Carnot o no}In the special case of Carnot groups, the
\emph{uniqueness} part of our result was already known, after \cite[Thm.
6.5]{BLUpaper}, and in our proof (Step III) we have explicitly exploited that
result, relying on the assumption (H3) and the lifting technique. On the other
hand, in the proof of our existence result (Steps I-II) we have never
exploited assumption (H3) and the lifing technique. Actually, our proof in
Steps I-II works also in Carnot groups, and our existence result extends the
one proved in \cite[Corollary 6.2]{BLUpaper}, where a stronger pointwise
(instead of integral) bound was assumed on $f$.
\end{remark}
If the initial datum satisfies a slightly stronger assumption than
\eqref{eq.growthf}, we can refine the previous results getting existence and
uniqueness of the solution for every $t>0$:
\begin{proposition} \label{prop.growthfless2}
 Let $f\in C(\mathbb{R}^{n})$ satisfy
 the growth assumption \eqref{eq.growthf} in the following stronger form:
 \emph{there exist $\alpha\in(0,2)$ and $\mu>0$ such that}
 \begin{equation}\label{eq.growthfstronger}%
 \int_{\mathbb{R}^{n}}\vert f(y)\vert\,\exp\big(-\mu\rho_{X}^{\alpha}(y)\big)\,\d y<+\infty. 
\end{equation}
Then, the function $u$ defined by \eqref{eq.defiusolution} is a classical
solution of \eqref{eq.mainCauchyPbHt} on 
$S_{\infty}:=(0,+\infty)\times\mathbb{R}^{n}$.
\end{proposition}
\begin{proof}
Using assumption \eqref{eq.growthfstronger}, it is easy to see that for
\emph{e\-ve\-ry} fixed $\theta>0$ one has%
\begin{equation} \label{eq.growthftheta}%
\int_{\mathbb{R}^{n}}\vert f(y)\vert\,\exp
 \big(-\theta\rho_{X}^{2}(y)\big)\,\d y<+\infty.
\end{equation}
As a consequence, from Theorem \ref{thm.mainExistenceCauchy} we derive that
the function $u$ in \eqref{eq.defiusolution} is a classical solution of
\eqref{eq.mainCauchyPbHt} on $S_{T/\theta}$ for every $\theta>0$,
hence on the whole of $S_{\infty}$.
\end{proof}
\section{An application to the Dirichlet problem for $\mathcal{H}$} \label{sec.appDirichlet}
 The aim of this section is to show how our global Gaussian estimates for
 $\Gamma$ can be used to study the solvability of the $\mathcal{H}$-Dirichlet
 problem on an \emph{arbitrary bounded domain} 
 $\Omega\subseteq\mathbb{R}^{1+n}$. 
 All the results we are going to present basically follow by combining
 the results of the previous sections with the investigations carried out 
 (in an abstract framework) in
 \cite{Kogoj, LancoUguz, LancTralUguz, TralliUguz}. \medskip
 
 To begin with, we need to establish the following proposition.
 \begin{proposition} \label{prop.SegmentProperty}
 The CC distance $d_{X}$ associated with our system
 $X=\{X_{1},\ldots,X_{m}\}$ of ho\-mo\-ge\-neous H\"{o}rmander's vector fields
 satisfies the so-called \emph{segment property:} for every fixed
 $x,y\in\mathbb{R}^{n}$ there exists a continuous path 
 $\gamma:[0,1]\rightarrow\mathbb{R}^n$ such that 
 $\gamma(0)=x,\,\gamma(1)=y$ and
 $$d_X(x,y) = d_X\big(x,\gamma(t)\big)+d_X\big(\gamma(t),y\big)\qquad\text{for
 all $0\leq t\leq 1$}.$$
\end{proposition}
 \begin{proof}
 This fact has been proved in \cite[Corollary 5.15.6]{BLUlibro} in
 the context of Carnot groups. Actually, the same proof can be repeated in our
 setting; the only nontrivial point that must be checked is that 
 the $d_X$-balls $B_{X}(x,\rho)$ are bounded in the
 Euclidean sense (for all $x\in\R^n$ and $\rho > 0$). \vspace*{0.07cm}
 
 To prove this fact, we argue as follows. First of all, since the distance
 $d_{X}$ is topologically equivalent to the Euclidean distance, there exists
 some $r > 0$ such that the Euclidean
 ball $B_{E}(0,1)$ contains the $d_X$-ball $B_{X}(0,r)$. 
 On the other hand, for every $R>0$ we have
 $$
 \delta_{r/R}\big(B_{X}(0,R)\big) = 
 B_{X}\left(
 0,r\right) \subseteq B_{E}(0,1);
 $$
 hence, $\delta_{r/R}(B_{X}(0,R))$ is
 bounded in the Euclidean sense and, by the explicit form of $\delta_{r/R}$, the
 same is true for $B_{X}(0,R)$. From this, since for any $x\in\R^n$ and
 $\rho > 0$ we have
 $B_{X}(x,\rho)  \subseteq
 B_{X}(0,R)$, with $R=\rho+d_{X}(x,0)$, we conclude that
 every $d_X$-ball is bounded in the Euclidean sense.
\end{proof}
 Using the segment property of $d_X$, jointly with the properties
 of $\Gamma$ listed in Theorem \ref{ThmC} and the global Gaussian estimates
 \eqref{eq.GuassianGamma}
 in Theorem \ref{thm.mainGaussianGamma}, we can apply
 the axiomatic approach developed in \cite{LancoUguz}: denoting by
 $H$ the sheaf of functions defined as
 $$\Omega\mapsto H(\Omega) := \big\{u\in C^\infty(\Omega):\,\text{$\mathcal{H}u = 0$
 in $\Omega$}\big\},
 $$ 
 we have
 that $(\R^n,H)$ is a $\beta$-harmonic space satisfying the Doob convergence property.
 In this con\-text, given a fixed open set $\Omega\subseteq\R^{1+n}$, we say that
 \begin{itemize}
  \item a function $u:\Omega\to\R$ is \emph{$\mathcal{H}$-harmonic}
  in $\Omega$ if $u\in H(\Omega)$;
  \item a function $u:\Omega\to (-\infty,+\infty]$ is 
  \emph{$\mathcal{H}$-superharmonic in $\Omega$} if
  \begin{itemize} 
   \item[(a)] $u$ is lower semi-continuous (l.s.c., for short) in $\Omega$;
   \item[(b)] the set $\{x\in\Omega:\,u(x)<+\infty\}$ is dense in $\Omega$;
   \item[(c)] for every $v\in C(\overline{\Omega})$ such that $v\big|_{\Omega}\in H(\Omega)$
   and $v\leq u$ on $\de\Omega$ one has $v\leq u$ on $\Omega$.
 \end{itemize}
  \item a function $u:\Omega\to [-\infty,+\infty)$ is 
  \emph{$\mathcal{H}$-subharmonic in $\Omega$} if
  $-u$ is $\mathcal{H}$-superharmonic in $\Omega$.
 \end{itemize}
 We denote by $\overline{H}(\Omega)$ (resp.\,$\underline{H}(\Omega)$) 
 the (convex) cone of the $\mathcal{H}$-superharmonic 
 (resp.\,$\Ht$-subharmonic) functions
 in $\Omega$. Obviously, we have $\underline{H}(\Omega) = -\overline{H}(\Omega)$ and
 $\overline{H}(\Omega)\cap\underline{H}(\Omega) = H(\Omega)$. \medskip
 
 Let now $\Omega\subseteq\R^{1+n}$ be a fixed open set, and let
 $\varphi\in C(\de\Omega)$. We say that a function $u:{\Omega}\to\R$
 is a \emph{classical solution} of the $\Ht$-Dirichlet problem
 \begin{equation} \label{eq.DirPbHt}
  \begin{cases}
 \Ht u = 0 & \text{in $\Omega$}, \\
 u\big|_{\de\Omega} = \varphi
 \end{cases}
 \end{equation}
 if it satisfies the following properties:
 \begin{itemize}
  \item $u\in C(\overline{\Omega})$ and $u\big|_{\Omega} \in C^2(\Omega)$;
  \item $\Ht u = 0$ in $\Omega$ and $u\big|_{\de\Omega} = \varphi$.
 \end{itemize}
 Since $\Ht = \LL - \de_t$ satisfies the Weak Maximum Principle
 on every open subset of $\R^{1+n}$ (see, e.g., \cite[Example 8.20]{BBBook}),
 there exists at most one classical solution of the Dirichlet problem \eqref{eq.DirPbHt};
 however, the \emph{existence} of such a solution for a general $\varphi\in C(\de\Omega)$
 is not guaranteed.
 For this reason, we in\-tro\-du\-ce the so-called 
 \emph{Perron–Wiener–Brelot–Bauer} (PWBB, in short) solution
 of \eqref{eq.DirPbHt}. \vspace*{0.07cm}
 
 Following \cite{LancoUguz}, we first consider the functions
 \begin{align*}
  &  \overline{H}^\Omega_\varphi(x) :=
   \inf\big\{u(x):\,\text{$u\in \overline{H}(\Omega)$ and $\liminf_{\omega\to\omega_0}u(\omega)
   \geq\varphi(\omega_0)$ for all $\omega_0\in\de\Omega$}\big\} \qquad\text{and} \\
   & 
   \underline{H}^\Omega_\varphi(x) :=
   \sup\big\{u(x):\,\text{$u\in \underline{H}(\Omega)$ and $\limsup_{\omega\to\omega_0}u(\omega)
   \leq\varphi(\omega_0)$ for all $\omega_0\in\de\Omega$}\big\}.
 \end{align*}
 Then, since $(\R^{1+n},H)$ satisfies Doob'sconvergence property, it can be proved that
 $$\overline{H}^\Omega_\varphi \equiv \underline{H}^\Omega_\varphi =:
 H^\Omega_\varphi\in H(\Omega).$$
 We shall call this function
 the PWBB solution of \eqref{eq.DirPbHt}.
 Obviously, if $u$ is the classical solution of \eqref{eq.DirPbHt}, one has
 $u\equiv H_\varphi^\Omega$ on $\Omega$; on the other hand, even if
 $H_\varphi^\Omega$ can be constructed for an arbitrary $\varphi\in C(\de\Omega)$
 and it is always $\Ht$-harmonic in $\Omega$, one cannot expect (in general) that
 $$\lim_{\omega\to \omega_0}H^\Omega_\varphi(\omega) = \varphi(\omega_0) \qquad
 \text{for $\omega_0\in\de\Omega$}.$$
 The following definition is thus plainly justified.
 \begin{definition} \label{def.regularPoint}
  A point $\omega_0\in\de\Omega$ is called \emph{$\Ht$-regular} if
  \begin{equation} \label{eq.regularPoint}
   \lim_{\omega\to \omega_0}H^\Omega_\varphi(\omega) = \varphi(\omega_0) \qquad
 \text{for all $\varphi\in C(\de\Omega)$}.
 \end{equation}
 \end{definition}
 Due to the validity of the segment property for $d_X$, the \alt96 good' behavior 
 of
 $\Gamma$ in Theorem \ref{ThmC}, and the validity
 of \emph{global} Gaussian estimates for $\Gamma$, we are entitled
 to apply to our context all the abstract results established
 in \cite{Kogoj, LancoUguz, LancTralUguz, TralliUguz}. 
 As a consequence, we obtain several
 necessary/sufficient conditions for a point $\omega_0\in\de\Omega$ to be
 $\mathcal{H}$-regular (in the sense of Definition \ref{def.regularPoint}). \medskip
 
 Throughout the sequel,
 given any compact set $K\subseteq\R^{1+n}$, we define
 \begin{equation} \label{eq.defHBalayageGeneralK}
 \begin{split}
  & V_K(\omega) = \liminf_{z\to\omega}\big(W_K(z)\big), \qquad \text{where} \\
  & \text{$W_K(z):= \inf\big\{v(z):\,v\in\overline{H}(\R^n)$,\,$v\geq 0$ on $\R^{1+n}$
  and $v\geq 1$ on $K$}\big\}.
  \end{split}
  \end{equation}
  The function $V_K$ is usually referred to as
  the $\Ht$-balayage of
 $u_0\equiv 1$ on $K$.
 \begin{theorem}\protect{\cite[Thm.s 4.6 and 4.11]{LancoUguz}} \label{thm.ExtConeLancoUguz}
  Let $\Omega\subseteq\R^{1+n}$ be \emph{a bounded open set}, and
  let $\omega_0 = (t_0,x_0)$ be a fixed point of $\de\Omega$.
  For any $r > 0$, we define 
  $$\Omega'_r(\omega_0) :=
  \big\{\omega = (t,x)\in\R^{1+n}\setminus\Omega:\,t\leq t_0,\,\,
  \big(d_X(x,x_0)^4+|t-t_0|^2\big)^{1/4}\leq r\big\},$$
  and we denote by $V_r$ the so-called 
  $\Ht$-balayage of $u_0\equiv 1$ on $\Omega'_r(\omega_0)$, that is,
  \begin{equation} \label{eq.defVrBalayage}
  V_r := V_{\Omega'_r(\omega_0)}.
  \end{equation}
  Then, following assertions are equivalent:
  \begin{itemize}
   \item $\omega_0$ \emph{is not} $\Ht$-regular;
   \item there exists $r > 0$ such that
	$V_r(\omega_0) < 1$;
   \item $V_r(\omega) \to 0$ as $r\to 0^+$.
  \end{itemize}
  On the other hand, if there exist
  real constants
  $M,\rho,\theta > 0$ such that
  $$\big|\big\{x\in \overline{B_X(x_0,M\rho)}:\,(t_0-\rho^2,x)\notin\Omega\big\}\big|
  \geq \theta\,|B_X(x_0,M\rho)|,$$
  then $\omega_0$ is $\Ht$-regular. 
 \end{theorem}
 Another \emph{sufficient} condition for $\Ht$-regularity
 is the following.
 \begin{theorem}\protect{\cite[Theorem 5.1]{Kogoj}}
  Let $\Omega\subseteq\R^{1+n}$ be an open set, and let
  $\omega_0 = (t_0,x_0)\in\de\Omega$ be fixed.
  Moreover, let 
  $\{B_\lambda\}_{0<\lambda<1}$ 
  be a basis of closed neighborhoods of $x_0$ in $\R^n$
  such that 
  $$\text{$B_\lambda\subseteq B_\mu$ if $0<\lambda < \mu\leq 1$}.$$ 
  For every
  $\lambda\in (0,1)$, we define
  $$\Omega_\lambda^c(\omega_0) := \big([t_0-\lambda,t_0]
  \times B_\lambda\big)\setminus\Omega\quad\text{and}\quad
  T_\lambda(\omega_0) := \big\{x\in\R^n:\,(t_0-\lambda,x)\in\Omega_\lambda^c(\omega_0)\big\}.$$
  Then the point $\omega_0$ is $\Ht$-regular if
  $$
  \limsup_{\lambda\,\searrow\,0^+}\int_{T_\lambda(\omega_0)}
  \gamma(\lambda,x_0,\xi)\,\d\xi > 0.$$
 \end{theorem}
 By making use of the so-called $\Ht$-Wiener function
 (associated with the open set $\Omega$ and the point
 $\omega_0\in\de\Omega$), it is possible to derive
 a \emph{necessary and sufficient} condition
 for $\omega_0$ to be regular.
 \begin{theorem}\protect{\cite[Theorem 5.4]{LancoUguz}} \label{thm.WienerFunct}
  Let $\Omega\subseteq\R^{1+n}$ be a bounded open set, and let
  $\omega_0 = (t_0,x_0)\in\de\Omega$ be fixed. 
  Moreover, given a number $p > 0$ and a sequence $\{r_k\}_{k\in\N}$ converging to 
  $0$ as $k\to+\infty$, we define
  the \emph{$\Ht$-Wiener function} \emph{(}associated with $\Omega$ and $\omega_0$\emph{)}
  as
  \begin{equation} \label{eq.defiFunctionW}
   \mathcal{W}(\omega) := \sum_{k = 1}^{+\infty}\frac{1-V_k(\omega)}{p^k},
   \end{equation}
  where $V_k = V_{r_k}$ and, for every $r > 0$,
  the function $V_r$ is as in \eqref{eq.defVrBalayage}. Then 
  $$\text{$\omega_0$ is $\Ht$-regular if and only if $
  \mathcal{W}(\omega)\to 0$ as $\omega\to\omega_0$}.$$
 \end{theorem}
 Finally, by making explicit use of our global Gaussian estimates
 for $\Gamma$, we can obtain criteria for $\Ht$-regularity which are resemblant to the
 classical results proved by Wiener and Landis for the heat operator
 $\Delta-\de_t$. 
 In order to clearly state these criteria, we first fix some
 notation. \medskip
 
 Given a compact set $K\subseteq\R^{1+n}$, let $V_K$ be the
 $\Ht$-balayage of $u_0\equiv 1$ on $K$ defined in
 \eqref{eq.defHBalayageGeneralK}. By classical results
 of Potential Theory, it is known that $V_K$
 is $\Ht$-superharmonic on $\R^{1+n}$; as a consequence,
 there exists
 a unique positive Radon measure $\mu = \mu_K$ on $\R^{1+n}$ such that
 $$\text{$\Ht V_k = -\mu_K$\,\,in $\mathcal{D}'(\R^{1+n})$\qquad and \qquad
 $\mathrm{supp}(\mu_K) = K$}.$$
 (see, e.g., \cite{NegScor}). We then define
 the \emph{$\Ht$-capacity} of $K$ as follows
 $$\mathcal{C}_{\Ht}(K) := \mu_K(K).$$
 Moreover, if $\mathcal{M}^+(K)$ denotes
 the set of non-negative Radon measures on $\R^{1+n}$ with support
 contained in $K$, we also define
 the \emph{$a$-Gaussian capacity} of $K$ as follows
 $$\mathcal{C}_a(K) := \sup\bigg\{\nu(K):\,
 \nu\in\mathcal{M}^+(K)\,\,\text{and}\,\,
 \int_{K}G_a(t,x;s,y)\,d\mu(s,y)\leq 1\,\,\text{for all $(t,x)\in R^{1+n}$}\bigg\},$$
 where for every $a > 0$ we have used the notation
 \begin{equation} \label{eq.aGaussianFun}
 G_a(t,x;s,y) := \begin{cases}
 0, & \text{if $t\leq s$}, \\[0.1cm]
 \displaystyle
 \frac{1}{|B_X(x,\sqrt{t-s})|}\,\exp\bigg(-a\,
 \frac{d_X^2(x,y)}{t-s}\bigg), & \text{if $t > s$}.
 \end{cases}
 \end{equation}
 Notice that, using \eqref{eq.aGaussianFun}, our Gaussian estimates 
 \eqref{eq.GuassianGamma} reads as
 $$\frac{1}{\varrho}\,G_{\varrho}(t,x;s,y)\leq
 \Gamma(t,x;s,y)\leq \varrho\,G_{1/\varrho}(t,x;s,y) \qquad
 (\text{for all $(t,x),(s,y)\in\R^{1+n}$}).$$
 Here is a \alt96 Wiener-type' test for $\Ht$-regularity.
 \begin{theorem}\protect{\cite[Theorem 1.1]{LancTralUguz}} \label{thm.WienerTypeLancoTralli}
   Let $\Omega\subseteq\R^{1+n}$ be a bounded open set, and let
   $\omega_0 = (t_0,x_0)\in\de\Omega$.
   For every fixed $\lambda\in(0,1)$ and
   every $h,k\in\N$, we define
   \begin{align*}
   \Omega_k^h(\omega_0,\lambda) & :=
   \bigg\{\omega = (t,x)\in\R^{1+n}\setminus\Omega:\,\lambda^{k+1}\leq t_0-t\leq\lambda^k, \\[0.1cm]
   & \qquad\quad \frac{1}{\lambda^{h-1}}\leq\exp\bigg(\frac{d_X^2(x_0,x)}{t_0-t}\bigg)
   \leq \frac{1}{\lambda^h},\,\,
   \big(d_X(x,x_0)^4+|t-t_0|^2\big)^{1/4}\leq \sqrt{\lambda}\bigg\}.
   \end{align*}
   Then, if $\varrho > 0$ is as in \eqref{eq.GuassianGamma},
   the following facts hold.
   \begin{itemize}
    \item if there exist $0<a\leq 1/\varrho$ and $b > \varrho$ such that
    $$\sum_{h,k = 1}^{+\infty}\frac{\mathcal{C}_{a}\big(
    \Omega_k^h(\omega_0,\lambda)\big)}{\big|
    B_X(x_0,\lambda^{k/2})\big|}\,\lambda^{bh} = +\infty,$$
    then the point $\omega_0$ is $\Ht$-regular.
    \item If the point $\omega_0$ is $\Ht$-regular, then
    $$\sum_{h,k = 1}^{+\infty}\frac{\mathcal{C}_{b}\big(
    \Omega_k^h(\omega_0,\lambda)\big)}{\big|
    B_X(x_0,\lambda^{k/2})\big|}\,\lambda^{ah} = +\infty,$$
    for every $0<a\leq1/\varrho$ and $b \geq \varrho$.
   \end{itemize}
 \end{theorem}
 Finally, a `Landis-type' condition for $\Ht$-regularity
 is given by the following theorem.
 \begin{theorem}\protect{\cite[Theorem 1.3]{TralliUguz}} \label{thm.LandisTralliUguz}
  Let $\Omega\subseteq\R^{1+n}$ be a bounded open set, and let
   $\omega_0 = (t_0,x_0)\in\de\Omega$.
   For every fixed $\lambda\in (0,1)$ and every $k\in\N$, we consider the set
   $$\Omega_{k}^c(\omega_0) :=
   \bigg\{\omega = (t,x)\in\R^{1+n}\setminus\Omega:\,
   \frac{1}{\lambda^{k\log(k)}}\leq \Gamma(t_0,x_0;t,x)\leq 
   \frac{1}{\lambda^{(k+1)\log(k+1)}}\bigg\}\cup\{(t_0,x_0)\},$$
   where $\Gamma$ is the global heat kernel of $\Ht$. Then 
   $\omega_0$ is $\Ht$-regular if and only if
   $$\sum_{k = 1}^{+\infty}V_{\Omega_{k}^c(\omega_0)}(\omega_0) = +\infty,$$
   where $V_{\Omega_{k}^c(\omega_0)}$ is the $\Ht$-balayage of
   $u_0\equiv 1$ on $\Omega_{k}^c(\omega_0)$, see \eqref{eq.defHBalayageGeneralK}.
 \end{theorem}
\section{Scale-invariant Harnack inequality for $\mathcal{H}$\label{sec.Harnack}}
In this last section we prove a scale-invariant Harnack inequality for
non-negative solutions of $\mathcal{H}u=0$. This fact easily follows, via the
lifting procedure, from the analogous result proved on Carnot groups in
\cite[Corollary 4.5]{BLUpaper}. It is however a result which is worthwhile to
be pointed out. \medskip

Given any point $\omega_{0}=(t_{0},x_{0})\in(0,+\infty)\times\mathbb{R}^{n}$,
any number $r>0$, we define
$$
C(\omega_{0},r):=
 \big\{(t,x)\in\mathbb{R}^{1+n}:\,d_{X}(x,x_{0})<r,\,\,|t-t_{0}|<r^{2}\big\}.
$$
Furthermore, for every $\lambda\in(0,1/2)$, we set
$$
S_{\lambda}(\omega_{0},r) := 
 \big\{(t,x)\in\mathbb{R}^{1+n}:\,d_{X}(x,x_{0})<(1-\lambda)r,
 \,\,\lambda r^{2}<t_{0}-t<(1-\lambda)r^{2}\big\}.
$$
We are ready to state our result.
\begin{theorem} \label{thm.Harnack}
For every $h,\,k=0,1,2,...$ and every fixed $\lambda
\in(0,1/2)$, it is possible to find a positive constant 
$\mathbf{\nu}=\mathbf{\nu}_{h,k,\lambda}>0$ such that, for every 
$\omega_{0}=(t_{0},x_{0})\in(0,+\infty)\times\mathbb{R}^{n}$, every $r>0$, and every nonnegative
function $u\in C^{2}(C(\omega_{0},r))$ satisfying $\mathcal{H}u=0$ on
$C(\omega_{0},r),$
\begin{equation} \label{eq.Harnack}%
\sup_{S_{\lambda}(\omega_{0},r)}\left\vert X_{i_{1}}\cdots X_{i_{h}}%
(\partial_{t})^{k}u\right\vert \leq\mathbf{\nu}\,r^{-(h+2k)}\,u(\omega_{0}),
\end{equation}
for every $i_{1},\ldots,i_{h}\in\{1,\ldots,m\}$.
\end{theorem}
\begin{proof}
Letting $v_{0}:=(x_{0},0)\in\mathbb{R}^{N}$ and $\widehat{\omega}_{0}%
:=(t_{0},v_{0})\in\mathbb{R}^{1+N}$, we define
\begin{align*}
&  \widehat{C}(\widehat{\omega}_{0},r):=\big\{(t,v)\in\mathbb{R}^{1+N}:
\,d_{\mathcal{Z}}(v,v_{0})<r,\,\,|t-t_{0}|<r^{2}\big\}
\qquad\text{and}\\[0.2cm]
&  \qquad\widehat{S}_{\lambda}(\widehat{\omega}_{0},r):=\big\{
(t,v)\in\mathbb{R}^{1+N}:\,{d}_{\mathcal{Z}}(v,v_{0})<(1-\lambda)r,\,\,\lambda
r^{2}<t_{0}-t<(1-\lambda)r^{2}\big\}  .
\end{align*}
Let then $u\in C^{2}(C(\omega_{0},r))$ be any non-negative function satisfying
of $\mathcal{H}u=0$ on $C(\omega_{0},r)$. Denoting by 
$\pi_{n}:\mathbb{R}^{N}\rightarrow\mathbb{R}^{n}$ the canonical 
projection of $\mathbb{R}^{N}$
onto $\mathbb{R}^{n}$, we set
$$
\widehat{u}(t,v):=u\big(t,\pi_{n}(v)\big)  \qquad(v\in\mathbb{R}^{N}).
$$
Since $B_{\mathcal{Z}}(v_{0},r)\subseteq\pi_{n}^{-1}\big(B_{X}(x_{0},r)\big)$ 
(see Proposition \ref{Prop properties distances}-(iii)), we have
$$
\widehat{u}\in C^{2}\big(\widehat{C}(\widehat{\omega}_{0},r)\big)  .
$$
Moreover, since $u\geq0$ and $\mathcal{H}u=0$ on $C(\omega_{0},r)$, from the
lifting property \eqref{eq.liftinHHG} we derive that
$$
\text{$\widehat{u}\geq0$\quad and \quad$\mathcal{H}_{\mathbb{G}}\widehat{u}=0$
\qquad on $\widehat{C}(\widehat{\omega}_{0},r)$}.
$$
Putting together these facts, we are entitled to 
apply \cite[Corollary 4.5]{BLUpaper}, obtaining
\begin{equation} \label{eq.HarnaksuG}%
\sup_{\widehat{S}_{\lambda}(\widehat{\omega}_{0},r)}\big\vert Z_{i_{1}}\cdots
Z_{i_{h}}(\partial_{t})^{k}\widehat{u}\big\vert \leq
\nu\,r^{-(h+2k)}\,\widehat{u}(t_{0},v_{0}), 
\end{equation}
where $\mathbf{\nu}>0$ is an absolute constant only depending on $h,k$ and
$\lambda$. We now claim that the above \eqref{eq.HarnaksuG} is precisely the
desired \eqref{eq.Harnack}. In fact, by the very definition of $\widehat{u}$,
we have
\begin{equation} \label{eq.vut0x0}%
\widehat{u}(t_{0},v_{0})=u(t_{0},x_{0})=u(\omega_{0}); 
\end{equation}
moreover, by repeatedly exploiting \eqref{eq.liftinZjXj}, we get
\begin{align*}
&  Z_{i_{1}}\cdots Z_{i_{h}}(\partial_{t})^{k}\widehat{u}(t,v)=
(\partial_{t})^{k}\Big(Z_{i_{1}}\cdots Z_{i_{h}}\big(v\mapsto u(t,\pi_{n}(v))\big)\Big) \\[0.2cm]
&  \quad=(\partial_{t})^{k}
  \Big(Z_{i_{1}}\cdots Z_{i_{h-1}}\big(v\mapsto
(X_{i_{h}}u)(t,\pi_{n}(v))\big)\Big) \\[0.2cm]
&  \quad=\ldots=\big(  (\partial_{t})^{k}X_{i_{1}}\cdots X_{i_{h}}u\big)
(t,\pi_{n}(v))\qquad\text{for all $(t,v)\in\widehat{C}(\widehat{\omega}_{0},r)$}.
\end{align*}
From this, taking into account that $\pi_{n}\left(  B_{\mathcal{Z}}%
(v_{0},(1-\lambda)r)\right)  =B_{X}(x_{0},(1-\lambda)r)$, we readily obtain
\begin{equation} \label{eq.supvsupuequal}%
\sup_{\widehat{S}_{\lambda}(\widehat{\omega}_{0},r)}\big\vert Z_{i_{1}}\cdots
Z_{i_{h}}(\partial_{t})^{k}\widehat{u}\big\vert = 
\sup_{S_{\lambda}(\omega_{0},r)}\big\vert X_{i_{1}}\cdots X_{i_{h}}
(\partial_{t})^{k}u\big\vert . 
\end{equation}
By combining \eqref{eq.HarnaksuG}, \eqref{eq.vut0x0} and
\eqref{eq.supvsupuequal}, we finally derive \eqref{eq.Harnack}, with an
absolute constant $\mathbf{\nu}>0$ which de\-pends on the chosen $h,\,k$ and
$\lambda$ (but not on $\omega_{0},\,r$ nor $u$). This ends the proof.
\end{proof}

\end{document}